\documentclass[a4paper]{article}

\usepackage[utf8]{inputenc}
\usepackage[T1]{fontenc}

\usepackage{amsmath}
\usepackage{amssymb}
\usepackage{hyperref}
\usepackage{cite}

\usepackage{tikz-cd}
\usetikzlibrary{positioning,arrows,decorations.markings}
\tikzset{
block/.style={
  draw, 
  rectangle, 
  minimum height=1.5cm, 
  minimum width=2.5cm, align=center,
  fill=blue!20
  }, 
line/.style={->,>=latex'}
}
\tikzset{negated/.style={
      decoration={markings,
           mark= at position 0.5 with {
               \node[transform shape] (tempnode) {$\backslash\!\!\backslash$};
            }
       },
       postaction={decorate}
    }
}

\newcommand{\dR}{\mathbb{R}}
\newcommand{\dC}{\mathbb{C}}
\newcommand{\K}{\mathbb{K}}
\newcommand{\Vs}{\mathcal{V}_{\rm sys}}
\newcommand{\Hc}{\mathcal{H}}
\newcommand{\Sc}{\mathcal{S}}
\newcommand{\Vsx}{\mathcal{V}_{\rm sys}^x}
\newcommand{\Vsh}{\hat{\mathcal{V}}_{\rm sys}}
\newcommand{\ima}{{\rm Im}\,}
\DeclareMathOperator{\re}{Re}
\DeclareMathOperator{\ran}{ran}

\DeclareMathOperator{\rk}{rk}

\DeclareMathOperator{\di}{diag}

\newenvironment{smallbmatrix}
{\left[\begin{smallmatrix}}
{\end{smallmatrix}\right]}

\renewcommand{\theta}{\vartheta}
\renewcommand{\phi}{\varphi}

\usepackage{pgf,tikz,pgfplots}
\pgfplotsset{compat=1.14}
\usepackage{mathrsfs}
\usetikzlibrary{arrows}

\usepackage[amsmath,amsthm,thmmarks]{ntheorem}

\theoremstyle{plain}
\newtheorem{definition}{Definition}[section]

\newtheorem{proposition}[definition]{Proposition}
\newtheorem{lemma}[definition]{Lemma}
\newtheorem{corollary}[definition]{Corollary}
\newtheorem{example}[definition]{Example}
\newtheorem{remark}[definition]{Remark}

\title{The difference between port-Hamiltonian, passive and positive real descriptor systems
 \thanks{\textbf{Acknowledgments:} HG gratefully acknowledges the funding within the SPP1984 ``Hybrid and multimodal energy systems'' by the Deutsche Forschungsgemeinschaft (DFG). DH acknowledges funding from the DFG within SFB910 ``Control of self-organizing nonlinear systems: Theoretical methods and concepts of application''. 
KC acknowledges funding from ProFIT (co-financed by the Europäischen Fonds für regionale Entwicklung (EFRE)) within the WvSC project: EA 2.0 - Elektrische Antriebstechnik.
The authors would like to thank Volker Mehrmann for proposing this research topic and for his valuable suggestions on an earlier draft of this manuscript, as well as Arjan van der Schaft and Philipp Schulze for valuable discussions.}
        }

\author{Karim Cherifi \thanks{Institut für Mathematik, Technische Universität Berlin, Stra\ss e des 17.\ Juni 136, 10623 Berlin, Germany ({\tt \{cherifi,gernandt, hinsen\}@math.tu-berlin.de}).} \and
Hannes Gernandt$^\dagger$ 
        \and Dorothea Hinsen$^\dagger$  }

\begin{document}

\maketitle

\begin{abstract}
The relation between passive and positive real systems has been extensively studied in the literature. In this paper, we study their connection to the more recently used notion of port-Hamiltonian descriptor systems. It is well-known that port-Hamiltonian systems are passive and that passive systems are positive real. Hence it is studied under which assumptions the converse implications hold. Furthermore, the relationship between passivity, KYP inequalities and a~finite available storage is investigated.
\end{abstract}
\textit{Keywords:}  port-Hamiltonian systems,\ differential-algebraic equations,\ minimal realizations, passive systems, positive real systems, Kalman-Yakubovich-Popov inequality\vspace{0.5cm}\\

\textit{MSC 2010:} 34A09, 93C05 (primary), 
93B20, 
15A39 (secondary)



\section{Introduction}
Port-Hamiltonian (pH) systems have been increasingly used in recent years as a unified structured framework for energy based modeling of systems, see e.g.\ \cite{BeaMXZ18,BeaMV19,JacZ12,UnM22,OrtSMM01,Sch04,SchJ14}. The pH formulation has gained interest from engineers and mathematicians due to its modeling flexibility and robustness properties \cite{BeaMV19, MehMS16,MehV20, MehMW20}. Specifically, pH systems are used in coupled networks of systems and multiphysics simulation and control. System coupling often imposes additional algebraic constraints on the system which naturally lead to \emph{linear time-invariant descriptor systems} in state-space form presented as
\begin{align}\label{DAE}
\begin{split}
		\tfrac{d}{dt} Ex(t) &= Ax(t) + Bu(t),\quad  x(0) = x_0,\\
	y(t) &= Cx(t) + Du(t),
	\end{split}
\end{align}
where $u: \dR \to \K^m$, $x: \dR \to \K^n$, $y : \dR \to \K^m$ are the \emph{input}, \emph{state}, and \emph{output} of the system, and $E,A\in\K^{n\times n}, B\in\K^{n\times m}$, $C\in\K^{m\times n}$, $D\in\K^{m\times m}$, and $\K=\dR$ or $\K=\dC$. The system \eqref{DAE} will be concisely denoted by $\Sigma=(E,A,B,C,D)$ and throughout it is assumed that the pair $(E,A)$ is \emph{regular} which means that $\lambda E-A$ is invertible for some $\lambda\in\dC$. 

In addition, we consider the following notation: For a matrix $A\in \K^{n \times m}$ let  $A^{\top},A^H,A^{-H}$ denote the  transpose, conjugate transpose, inverse of $A^H$, respectively. Note that in the real case $A^\top = A^H$. The identity matrix of dimension $n$ is denoted by $I_n$. For a Hermitian matrix $A \in \K^{n \times n}$, we use $A > 0$ $(A \geq 0)$ if $A$ is positive (semi-) definite. Furthermore, we denote the set of eigenvalues of a matrix pencil $sE-A$ by 
\[
\sigma(E,A):=\{\lambda\in\dC~\mid~{\rm ker}(\lambda E-A)\neq\{0\}\}.
\]
Port-Hamiltonian (pH) systems are then defined as follows:

\begin{itemize}
\item[\rm (pH)] The system $\Sigma$ is  \emph{port-Hamiltonian} if there exists  $J,R,Q\in\K^{n\times n}$, $G,P\in\K^{n\times m}$, and  $S,N\in\K^{m\times m}$ such that 
\begin{align}
\label{def_PH}
\begin{split}
\begin{bmatrix}
A&B\\C&D
\end{bmatrix}&=\begin{bmatrix}(J-R)Q&G-P\\(G+P)^HQ&S+N\end{bmatrix},\quad Q^HE=E^HQ\geq 0,\\
	\Gamma &:= \begin{bmatrix}
		J & G \\
		-G^H& N\end{bmatrix}
		= - \Gamma^H,\quad 
 W := \begin{bmatrix}
	Q^HRQ & Q^HP\\
	P^HQ & S
\end{bmatrix} =W^H \geq 0.
\end{split}
\end{align}
\end{itemize}
Here the quadratic function $\Hc(x):=\frac{1}{2} x^HE^HQx$ is called the \emph{Hamiltonian} which can oftentimes be interpreted as the energy of the system. Note that recently in \cite{MehS22,MascvdSc18,GerHR21} also a geometric pH framework was developed which is based on monotone, Dirac and Lagrangian subspaces and enlarges the class of pH systems. Furthermore, some references assume that the matrix $Q$ in \eqref{def_PH} is positive definite. In this case the $Q$ in the matrix $W$ given by \eqref{def_PH} is oftentimes replaced by the identity.

It is well-known that pH descriptor systems satisfy the following dissipation inequality which is referred to as \emph{passivity} in the literature  \cite{MorM19}. 
\begin{itemize}
    \item[\rm (Pa)] The system $\Sigma$ is \emph{passive} if there exists $Q\in\K^{n\times n}$ such that 
    $\Sc(x)=\tfrac{1}{2}x^H Q^HEx$, called \emph{storage function}, satisfies for all $T\geq 0$ the following inequality 
    \begin{align}
        \label{dissi_ineq}
    \Sc(x(T))-\Sc(x(0))\leq \int_{0}^{T} \re y(\tau)^Hu(\tau) d \tau,\quad \Sc(x(T))\geq 0,
        \end{align}
    \end{itemize}
    for all consistent initial values $x(0)=x_0$ and all functions $x,u,y$ whose derivatives of arbitrary order $k\in\mathbb{N}$ exist and fulfill \eqref{DAE}. 
    Although, we only consider in this paper smooth functions in the dissipation inequality \eqref{dissi_ineq}, it can easily be extended to inputs $u$ which are weakly differentiable up to some order by using the density of smooth functions in the Lebesgue space $L^1([0,T],\K^m)$, see e.g.\ \cite{ReiV19}. 
    


The property (Pa) is hard to verify in practice, since one would have to consider all possible solution trajectories. It is more convenient to solve a linear matrix inequality called KYP (discovered independently by Kalman, Yakubovich and Popov) which is equivalent to (Pa) and can be obtained by differentiation of \eqref{dissi_ineq}. An overview is presented in \cite[p.\ 81]{BroLME07} for \emph{standard systems}, i.e.\ $E=I_n$ and in \cite{FreJ04,ZhaLX02} for descriptor systems.

\begin{itemize}
\item[\rm (KYP)] The system $\Sigma$ whose matrices have entries in $\K$ has a  solution $Q\in\K^{n\times n}$ to the \emph{generalized KYP inequality} if  
\begin{align}
\label{eq:kyp_dae}
\begin{bmatrix}
-A^HQ-Q^HA& C^H-Q^HB\\C-B^HQ&D+D^H
\end{bmatrix}\geq 0,\quad E^HQ=Q^HE\geq 0.
\end{align}
\end{itemize}

In many applications only input-output data is given and hence an important question is whether we can decide if a system is pH from this data and even more, we want to obtain a pH representation \eqref{def_PH} of the system. The typical approach is to apply a Laplace transformation to \eqref{DAE} which leads to the \emph{transfer function} 
\begin{align}
\label{transfer}
\mathcal{T}(s):=C(sE-A)^{-1}B+D
\end{align}
that describes the input-output behavior in the frequency domain. It is well-known for standard  systems that the passivity implies that its transfer function is \emph{positive real}, see \cite{And67,AndVon73}.
\begin{itemize}
\item[\rm (PR)] The system $\Sigma$ with transfer function $\mathcal{T}$ given by \eqref{transfer} is called \emph{positive real} if $\mathcal{T}$ has no poles for all  $s\in\dC$, $\re s>0$ and satisfies $\mathcal{T}(s)+\mathcal{T}(s)^H\geq 0$ for all $s\in\dC$ and $\re s>0$. 
\end{itemize}

As mentioned above, it is well-known that pH descriptor systems are passive \cite{MorM19}, denoted by (Pa), and that passive systems are positive real (PR). Moreover, the passivity is implied by the existence of solutions to KYP inequalities, see Proposition \ref{prop:dae_equiv}. The main goal of this note is to investigate under which assumptions also the converse implications hold. The study of these implications require the use of controllability, observablility and minimality notions.

Given a transfer function $\mathcal{T}$, a \emph{realization} is finding the matrices $(E,A,B,C,D)$ in a descriptor state space form \eqref{DAE} such that \eqref{transfer} is satisfied. In addition, the realization is called \emph{minimal} if the number of states $n$ in \eqref{DAE} needed to represent $\mathcal{T}$ is minimal.

 A system is called  \emph{controllable} (\emph{observable}) if and only if
\begin{align}
    \label{def:contr}
\rk[\lambda I_n-A,B]=n \quad (\rk[(\lambda I_n-A)^\top,C^\top]=n),\quad  \text{for all}\, \lambda\in\dC.
\end{align}
In the case of standard systems minimality is equivalent to the system being both controllable and  observable.

Furthermore, one can define a weaker property \emph{stabilizability} (\emph{detectability}) of controllability (observability) such that
\begin{align}
\label{def:det}
\rk[\lambda I_n-A,B]=n\quad (\rk[(\lambda I_n-A)^\top,C^\top]=n),\quad  \text{for all}\, \lambda\in\dC,\, \re\lambda\geq 0.
\end{align}

In \cite{Dai89,VergLevyKail81} the conditions on minimality were generalized to descriptor systems by showing that a realization $(E,A,B,C,D)$ of a transfer function $\mathcal{T}(s)$ is minimal if and only if it fulfills the following conditions
\begin{align}
\label{DAE_minimal}
&\rk[\lambda E-A,B]=n,
 \quad \rk[(\lambda E-A)^\top,C^\top]=n,&& \text{for all}\, \lambda\in\dC,\\ 
&\rk[E,B]=\rk\begin{bmatrix}
E\\C
\end{bmatrix}=n, \quad A\ker E\subseteq\ran E.&& \label{DAE_minimal_2}
\end{align}
The first (second) property in \eqref{DAE_minimal} defines the \emph{behavioral controllability} (\emph{behavioral observability}) of the system. If the system fulfills in addition to behavioral controllability $\rk[E,B]=n$, it is called \emph{completely controllable}. If the system is behavioral observability and $\rk[E^\top,C^\top]=n$ holds, it is \emph{completely observable}.

For standard systems, i.e.\ $E=I_n$ in \eqref{DAE}, which are controllable and observable, it is well-known that (pH), (Pa), (KYP) and (PR) are equivalent. For uncontrollable or unobservable standard systems a detailed study between the relation between (Pa), (KYP) and (PR) has been conducted in \cite[Chapter 3]{BroLME07} but without including pH systems. The connection between (Pa) and (pH) was discussed in \cite{Geop09}. Another recent survey for standard systems was given in \cite{HugB21} see p.\ 59 therein for a discussion on unobservable and uncontrollable systems. An overview for standard systems is presented in Figure~\ref{fig:overvieweinvertible}.

    \begin{figure}
    \centering
        \begin{tikzpicture}
        \node[block] (a) {(pH)};
        \node[block, below =2cm of a]   (b){(KYP)};
        \node[block, right =2cm of b]   (c){(Pa)};
        \node[block, above =2cm of c]   (d){(PR)};
        
        \draw[->,semithick,double,double equal sign distance,>=stealth, color=blue] ([xshift=-2ex]a.south) -- ([xshift=-2ex]b.north);
        \draw[->,semithick,double,double equal sign distance,>=stealth,negated, color=red] ([xshift=2ex]b.north) -- ([xshift=2ex]a.south) node[midway,right = 2 ex]{Ex.~\ref{ex:onlyzero}~~~};
        \draw[->,semithick,double,double equal sign distance,>=stealth] (b.west) --  ++(-10pt,0) coordinate[yshift=-1.7 cm,](r){} |- (a.west)node[midway,below left= 0.75cm and 1ex]{\parbox{2cm}{\begin{center}$\ker Q$\\ $\subseteq$ \\ $\ker C\cap\ker A$\\ or\\ observable\end{center}}};
        \draw[->,semithick,double,double equal sign distance,>=stealth, color=blue] ([yshift=-2ex]b.east) -- ([yshift=-2ex]c.west);
        \draw[->,semithick,double,double equal sign distance,>=stealth, color=blue] ([yshift=2ex]c.west) -- ([yshift=2ex]b.east) node[midway,below = 2 ex]{};
        \draw[->,semithick,double,double equal sign distance,>=stealth, color=blue]([xshift=2ex]c.north) -- ([xshift=2ex]d.south);
        \draw[->,semithick,double,double equal sign distance,>=stealth,negated,color=red] ([xshift=-2ex]d.south) -- ([xshift=-2ex]c.north) node[midway,left = 2 ex]{Ex. \ref{ex:obsv!notcontr}~~};
        \draw[->,semithick,double,double equal sign distance,>=stealth, color=black] (d.east) --  ++(10pt,0) coordinate[yshift=-1.7cm,](r){} |- (c.east)node[midway,below right= -2 cm and 1ex]{controllable};
        \draw[->,semithick,double,double equal sign distance,>=stealth] ([yshift=2ex]d.west) -- ([yshift=2ex]a.east)node[midway,above = 2 ex]{minimal};
          \end{tikzpicture}
    \caption{The relationship between (pH), (KYP), (Pa) and (PR) for a standard system ($E= I$). The color blue marks implication without additional assumption and the color black implications with additional assumptions. The counterexamples, if assumptions are not fulfilled, are colored in red. }
    \label{fig:overvieweinvertible}
\end{figure}
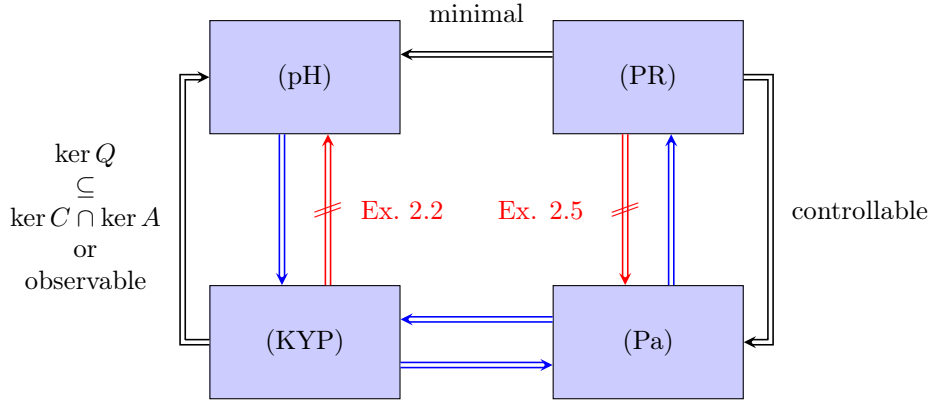  

For descriptor systems the relations between (pH), (Pa), (KYP) and (PR) were already studied in numerous works \cite{CamF09, FreJ04, Masu06, ReiRV15, ReiS10, ReiV15,ZhaLX02}. However, not all four properties have been investigated at the same time and oftentimes the minimality of the descriptor system is assumed.

As a first step, we combine the aforementioned results to obtain
\[
\text{(pH)}\quad \Longrightarrow\quad \text{(KYP)}\quad \Longrightarrow \quad \text{(Pa)}\quad \Longrightarrow\quad  \text{(PR)}
\]
which holds without observability or controllability assumptions and we  provide examples showing that the converse implications do not hold.

Our aim is to provide sufficient conditions for the converse implications to hold. Hence the remaining questions which we will answer are 
\begin{itemize}
    \item[\rm (Q1)] When do solutions to the KYP inequality lead to a pH formulation? 
    \item[\rm (Q2)] When does passivity lead to solutions of the KYP inequality and can we realize passive systems as pH systems?
    \item[\rm (Q3)] Can every positive real transfer function be realized as a pH system?
\end{itemize}

The answer to question (Q1) is related to observability properties of the system. For standard systems it was shown in \cite[p. 55]{Geop09} that only those solutions $Q\in\K^{n\times n}$ to the KYP inequality which additionally satisfy
\begin{align}
\label{ker_incl}
\ker Q\subseteq\ker A\cap\ker C
\end{align}
lead to a pH formulation. Conversely, the $Q$ used in (pH) automatically satisfies \eqref{ker_incl}. We show that the same condition is true for descriptor systems. If the system is behaviorally observable then $\ker A\cap\ker C=\{0\}$ and hence the existence of a pH formulation is equivalent to the existence of invertible solutions to the KYP inequality.

Due to the interesting properties of pH systems, ideally, we want to obtain a pH representation for any passive system given in DAE form or given time domain data measurements \cite{CheMH19}. In order to solve that problem we need to answer question (Q2). This problem was already studied in \cite{CamF09,ReiRV15} where it was shown that passivity only guarantees the KYP inequality to hold on certain subspaces and as a consequence, we can only derive a pH formulation on a subspace. However, if the system has index at most one we can derive a modified KYP inequality which is solvable for passive systems and which leads for behaviorally observable systems to a pH formulation of the system on the whole space.

The question (Q3) arises when one has to reconstruct a system from input-output data. If the system is expected to be passive, then the transfer function is positive real. If the data does not allow us to conclude (PR), e.g.\ due to measurement errors, then one can compute the nearest positive real transfer function \cite{GilS18}. Therefore, the remaining task is to find a pH state space representation \eqref{def_PH} of this positive real transfer function, i.e.\ one has to compute the system matrices $(E,A,B,C,D)$. We show that the computation of the system matrices is always possible for a given positive real transfer function and this is based on the well-known representation, see e.g.\ \cite[Section 5.1]{AndVon73}
\[
\mathcal{T}(s)=M_1s+\mathcal{T}_p(s)
\]
for some positive semidefinte $M_1\in\dR^{n\times n}$ and a proper positive real rational function $\mathcal{T}_p(s)$. The summand $M_1s$ can be realized as an index two subsystem which is combined with a pH realization based on a  minimal realization of $\mathcal{T}_p(s)$.

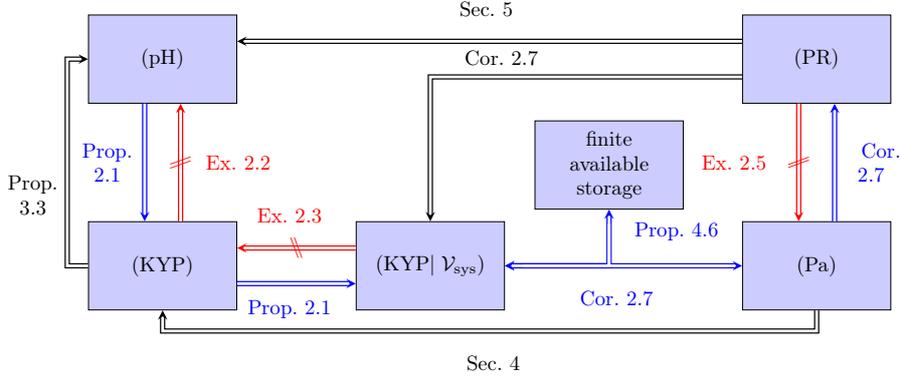
\begin{figure}
    \centering
    \scalebox{0.78}{
\begin{tikzpicture}
        \node[block] (a) {(pH)};
        \node[block, below =2cm of a]   (b){(KYP)};
        \node[block, right =2cm of b]   (c){(KYP$\mid\Vs$)};
        \node[block, right =4cm of c]   (d){(Pa)};
        \node[block, above =2cm of d]   (e){(PR)};
        \node[block, above right = 0.2cm and 0.5cm of c ]   (f){finite\\available\\storage};
        
        \draw[->,semithick,double,double equal sign distance,>=stealth, color=blue] ([xshift=-2ex]a.south) -- ([xshift=-2ex]b.north)node[midway,left = 0 ex]{\parbox{1cm}{\begin{center}Prop.\\~\ref{prop:dae_equiv}~~\end{center}}};
        \draw[->,semithick,double,double equal sign distance,>=stealth,negated, color=red] ([xshift=2ex]b.north) -- ([xshift=2ex]a.south) node[midway,right = 2 ex]{Ex.~\ref{ex:onlyzero}~~~};
        \draw[->,semithick,double,double equal sign distance,>=stealth, color=blue] ([yshift=-2ex]b.east) -- ([yshift=-2ex]c.west)node[midway,below = 1 ex]{Prop.~\ref{prop:dae_equiv}~~~};
        \draw[->,semithick,double,double equal sign distance,>=stealth,negated, color=red] ([yshift=2ex]c.west) -- ([yshift=2ex]b.east) node[midway,above = 2 ex]{Ex.~\ref{ex:dae_notkyp}~~~};
        \draw[<->,semithick,double,double equal sign distance,>=stealth,color=blue] (c.east) -- (d.west)node[midway,below = 2 ex]{Cor.~\ref{cor:restr}~~~};
        \draw[<->,semithick,double,double equal sign distance,>=stealth, color=blue] (c.east) -| (f.south)node[midway,above right = 2 ex and 2ex]{Prop.~\ref{prop: finit}~~~};
         \draw[-,semithick, white,line width=1.4pt, shorten >= 7pt] ([xshift=2ex]c.east) -- (d.west);
        \draw[-,semithick, white,line width=1.4pt, shorten >= 7pt] ([xshift=2ex]c.east) -| (f.south);
        \draw[->,semithick,double,double equal sign distance,>=stealth,negated, color=red] ([xshift=-2ex]e.south) -- ([xshift=-2ex]d.north)node[midway,left = 2 ex]{Ex.~\ref{ex:obsv!notcontr}~~};
        \draw[->,semithick,double,double equal sign distance,>=stealth, color=blue] ([xshift=2ex]d.north) -- ([xshift=2ex]e.south)node[midway,right = 1 ex]{\parbox{1cm}{\begin{center}Cor.\\~\ref{cor:restr}~~~~~\end{center}}} ;
        \draw[->,semithick,double,double equal sign distance,>=stealth] ([yshift=2ex]e.west) -- ([yshift=2ex]a.east)node[midway,above = 2 ex]{Sec.~\ref{sec:PR}~~};
        \draw[->,semithick,double,double equal sign distance,>=stealth, color=black] ([yshift=-2ex]e.west) -| (c.north)node[midway,above right= 0.5ex and 3ex]{Cor.~\ref{cor:restr}~~~};
        \draw[->,semithick,double,double equal sign distance,>=stealth] (d.south) |-  ++(0,-10pt) coordinate[yshift=-1.7cm](r){} -| (b.south)node[midway,below right= 2 ex and 5cm]{Sec.~\ref{sec:Pa_pH}~~};
        \draw[->,semithick,double,double equal sign distance,>=stealth] (b.west) --  ++(-10pt,0) coordinate[yshift=-1.7cm,](r){} |- (a.west)node[midway,below left= 1.5cm and 0.0ex]{\parbox{0.9cm}{\begin{center}Prop.\\~\ref{prop:arjan}~~\end{center}}};
        \end{tikzpicture}
        }
    \caption{Overview of the main results in this note. The implications with additional assumption are colored black and the one without in blue. Moreover, counterexamples, if assumptions are not fulfilled, are highlighted in red.}
    \label{fig:overview}
\end{figure}

The paper is organized as follows: In Section~\ref{sec:known} we summarize which relations between the basic notions (pH), (KYP), (Pa), (PR) are known for descriptor systems. In Section~\ref{sec:KYP_pH}, it is shown that \eqref{ker_incl} can be used to define a pH realization which answers (Q1). Besides that, a~possible generalization of the solutions of (KYP) is discussed. The answer to (Q1) will be used in Section~\ref{sec:Pa_pH} where we consider (Q2). In particular, we derive a pH formulation for passive index one systems which are (behaviorally) observable. 
Therein we further discuss extensions of the \emph{available storage} introduced in \cite{Wil72} to descriptor systems which was previously considered in \cite{CamF09,ReiS10} for observable and controllable systems. Finally, in Section \ref{sec:PR} we show how a pH formulation can be derived from a given real-valued positive real transfer function which answers (Q3). As a summary, an overview of the main results is presented in Figure~\ref{fig:overview}.

\section{Literature review and combination of known results}
\label{sec:known}
Below, we give the first main result on the relationship between (pH), (Pa), (KYP) and (PR) for descriptor systems. Here we combine the results of \cite{FreJ04}, who studied (PR) and (KYP), and \cite{GilS18} who studied the relation between (PR), (pH) and (KYP) for invertible $Q$. In addition, we include the relation to passivity. 
\begin{proposition}
\label{prop:dae_equiv}
Let $(E,A,B,C,D)$ define a linear time-invariant descriptor system \eqref{DAE}. Then the following holds
\begin{align}
\label{dae_implic}
\emph{(pH)} \quad  \Longrightarrow \quad \emph{(KYP}) \quad \Longrightarrow \quad  \emph{(PR)} ~~\land ~~ \emph{(Pa)}.
\end{align}
Furthermore, every $Q\in\K^{n\times n}$ fulfilling \emph{(pH)} is a solution to $\emph{(KYP})$ and every solution $Q$ to $\emph{(KYP})$  leads to a storage function in \emph{(Pa)}. Moreover, the following holds
\begin{align}
\label{KYP_to_pH}
\emph{(KYP}) ~ \text{with invertible $Q$} \quad &  \Longrightarrow   \quad \emph{(pH)},\\ \emph{(Pa)} \land \text{$E$ invertible} \quad & \Longrightarrow \quad \emph{(KYP}). \label{Pa_to_KYP}
\end{align}
\end{proposition}
\begin{proof}
\underline{\em Step 1:} We prove the implications \eqref{dae_implic}. 
If (pH) holds for some $Q$, it solves the KYP inequality \eqref{eq:kyp_dae} since $E^HQ=Q^HE\geq 0$ and
\begin{align}
&~~~\,\begin{bmatrix}
\nonumber
-A^HQ-Q^HA&& C^H-Q^HB\\C-B^HQ&&D+D^H
\end{bmatrix}\\&=\begin{bmatrix}
-Q^H(J-R)^HQ-Q^H(J-R)Q && Q^H(G+P)-Q^H(G-P)\\(G+P)^HQ-(G-P)^HQ&&2S
\end{bmatrix}\\&=2\begin{bmatrix}
Q^H&0\\0& I_m
\end{bmatrix}\begin{bmatrix}
R&P\\P^H&S
\end{bmatrix}\begin{bmatrix}
Q&0\\0& I_m
\end{bmatrix}=2W\geq 0. \label{W_Q}
\end{align}
Hence $Q$ fulfills (KYP). Next, we show that any $Q\in\K^{n\times n}$ which fulfills (KYP) defines a storage function $\Sc(x):=\tfrac{1}{2}x^HQ^HEx$ which fulfills (Pa). The basic idea goes back to \cite{Wil71} for standard systems. For sufficiently smooth $u$, consistent initial value $x_0$ and for all $t\geq 0$, it holds that
\begin{align}
\nonumber
2 \tfrac{d}{dt}\Sc(x(t))&=\tfrac{d}{dt}(Ex(t))^HQx(t)+x^H(t)Q^H\tfrac{d}{dt}Ex(t)\\&=(Ax(t)+Bu(t))^HQx(t)+x^H(t)Q^H(Ax(t)+Bu(t))\nonumber\\&=\begin{bmatrix}
x(t)\\ u(t)
\end{bmatrix}^H\begin{bmatrix}
A^HQ+Q^HA&&Q^HB-C^H\\B^HQ-C&&-D-D^H
\end{bmatrix}\begin{pmatrix}
x(t)\\ u(t)
\end{pmatrix}\nonumber\\&~~~~+\begin{bmatrix}
x(t)\\ u(t)
\end{bmatrix}^H\begin{bmatrix}
0&&C^H\\C&&D+D^H
\end{bmatrix}\begin{bmatrix}
x(t)\\ u(t)
\end{bmatrix}\nonumber \\ 
&\leq \begin{bmatrix}
x(t)\\ u(t)
\end{bmatrix}^H\begin{bmatrix}
0&&C^H\\C&&D+D^H
\end{bmatrix}\begin{bmatrix}
x(t)\\ u(t)
\end{bmatrix}=2\re y(t)^Hu(t). \label{pass_ungl}
\end{align}
Integration of \eqref{pass_ungl} leads to (Pa). Finally, it was shown in \cite[Theorem 3.1]{FreJ04} that (KYP) implies (PR). This completes the proof of \eqref{dae_implic}. \\

\underline{\em Step 2:} To prove \eqref{KYP_to_pH}, let $Q$ be an invertible solution of (KYP). Then we can define a pH system via
\begin{align*}
 J&:= \frac{1}{2} (A Q^{-1} - Q^{-H} A^H), &
	R&:=- \frac{1}{2} (A Q^{-1} + Q^{-H} A^H),\\ G &:= \frac{1}{2}(Q^{-H} C^H + B), &
	P&:= \frac{1}{2}(Q^{-H} C^H - B), \\ S&:= \frac{D + D^H}{2}, &
	 N&:= \frac{D - D^H}{2}.
\end{align*}
Hence by definition it holds that $-\Gamma=\Gamma^H$, $W =W^H$ and
\begin{align*}
	(J-R)Q = A,\quad  G-P = B, \quad  (G+P)^H Q = C,\quad S+N = D.
\end{align*}
Furthermore, $W\geq 0$ follows from \eqref{W_Q} and hence (pH) is satisfied. This proves \eqref{KYP_to_pH}.\\
\underline{\em Step 3:}  We continue with the proof of \eqref{Pa_to_KYP}.
Let $Q\in\K^{n\times n}$ be such that $\mathcal{S}(x):=\tfrac12x^HQ^HEx$ defines a storage function. 
To show (KYP) we verify the left  inequality in \eqref{eq:kyp_dae} first. Since $E$ is invertible the solutions of \eqref{DAE} are given by the solutions of the standard system 
$\dot x(t)=E^{-1}Ax(t)+E^{-1}B u(t)$, $x(0)=x_0$. In particular, for every choice of $(x_0,u(0))\in\K^{n}\times\K^m$ there exists a solution which fulfills \eqref{pass_ungl} as a consequence of (Pa). Using now $t=0$ in \eqref{pass_ungl} this implies the left inequality in (KYP). Furthermore, by choosing $T=0$ in (Pa) and the previous observation, that for every initial value  $x_0\in\K^n$ and $u=0$ there exists a solution we conclude from the  inequality on the right-hand side in (Pa) that $Q^HE\geq 0$ holds. In summary, this proves that $Q$ solves (KYP) and hence the implication \eqref{Pa_to_KYP}.
\end{proof}
However, we show with the following example that (KYP) does not necessarily imply (pH).
\begin{example}
\label{ex:onlyzero}
Consider the real system given by $(E,A,B,C,D)=(1,-1,1,0,0)$. Then this system is asymptotically stable i.e. all eigenvalues of matrix $A$ have negative real parts, hence using \eqref{def:contr} and \eqref{def:det}, we conclude that it is detectable and controllable but not observable. Furthermore, the KYP inequality \eqref{eq:kyp_dae} which is given by
\[
\begin{bmatrix}
2Q&-Q\\-Q&0
\end{bmatrix}\geq 0
\]
has only the trivial solution $Q=0$. Indeed, for $Q\neq 0$ the above matrix is indefinite.  
To construct a pH system notice that $(J-R)Q=A=-1$ which implies $R=Q^{-1}$ and 
 \[
B=G-P=1,\quad C=G+P=0,\quad D=S+N=0.
\]
This yields $G=\frac{1}{2}=-P$, $N=S=0$ and hence 
\[
\begin{bmatrix}
Q^\top RQ&&Q^\top P\\P^\top Q&&S
\end{bmatrix}=\begin{bmatrix}
R^{-1}&-\frac{1}{2}R^{-1}\\-\frac{1}{2}R^{-1}&0
\end{bmatrix}
\]
which is indefinite. Therefore the system is not pH.
\end{example}

Next, we present an example that shows that $\rm{(PR)}\land\rm{(Pa)}\nRightarrow\rm{(KYP)}$ for descriptor systems.
\begin{example}
\label{ex:dae_notkyp}
Consider the real system $(E,A,B,C,D)=(0,1,0,1,0)$. Then the system dynamics is given by $x(t)= 0$. Hence, $x_0=0$ is the only consistent initial value. Therefore for all $u\in L^1([t_0,t_1],\dR)$, $0\leq t_0\leq t_1$ we have 
\[
\int_{t_0}^{t_1}  y(\tau)u(\tau)d\tau=\int_{t_0}^{t_1}  x(\tau)u(\tau)d\tau=0,
\]
which implies passivity. On the other hand, the corresponding LMI of this system for some $x\in\dR$ is given by 
\[
\begin{bmatrix}
-2x&&1\\1&&0
\end{bmatrix}\geq 0.
\]
However, this cannot be valid since the matrix is indefinite for all $x\in\dR$. Furthermore, the system is positive real with transfer function $\mathcal{T}(s)=C(sE-A)^{-1}B+D=0$. Moreover, the system is behaviorally controllable but not minimal since $\rk[E,B]=0\neq 1$ and $A\ker E\nsubseteq\ran E$. The proof that the system is not pH is similar to the calculation in Example~\ref{ex:onlyzero}.
\end{example}
Note that it was important in the Examples \ref{ex:onlyzero} and \ref{ex:dae_notkyp} that the feedthrough term $D$ is zero. It is an open problem if a similar construction is possible with a non-trivial feedthrough term. \\[2ex]

Below, we briefly review known results: 
\begin{itemize}
\item[\rm (i)] In Theorem 2 of \cite{ZhaLX02}, the authors show that (KYP) with the condition $E^HQ\geq 0$ is equivalent to the so called extended strict positive realness and $D+D^H>0$ under the condition that the systems are assumed to have index one and be asymptotically stable. 
\item[\rm (ii)] In \cite{FreJ04}, it was shown that (KYP) implies (PR) and the converse is also true for certain minimal realizations of the system.
\item[\rm (iii)] In \cite{Masu06} it was shown that a nonnegative \emph{Popov function} leads to a  Hermitian solution to the KYP inequality \eqref{eq:kyp_dae}. It is assumed that there are no eigenvalues on the imaginary axis and that the system has index one.
\item[\rm (iv)] The author of  \cite{Hugh17} studies a more general behavioral approach and shows equivalence of passivity and positive real pairs associated to the behavior. Although this approach contains descriptor systems, they were not investigated explicitly.
\item[\rm (v)] In \cite{GilS18}, the authors have shown that (pH) $\Rightarrow$ (KYP) $\Rightarrow$ (PR) and that (KYP) $\Rightarrow$ (pH) holds for invertible solutions $Q$ of the KYP inequality. However, it remains open whether non-invertible solutions to the KYP inequality also lead to a pH formulation or if (PR) also implies (KYP).
\item[\rm (vi)] In \cite[Section II]{MorM19}, it is proven that the implication (pH) $\Rightarrow$ (Pa) also holds for descriptor systems whose coefficients depend on the time $t$ and the state $x$ and not necessarily quadratic Hamiltonians.
\item[(vii)] In \cite{GerH21}, the authors consider \emph{stable} systems, i.e.\ for all consistent initial values $x_0\in\K^n$ the solution $x$ to $E\dot x(t)=Ax(t)$, $x(0)=x_0$ fulfills $\sup_{t\geq 0}\|x(t)\|<\infty$. It was shown that there exists $Q\in\K^{n\times n}$ based on the solution of a generalized Lyapunov equation such that $A=(J-R)Q$ holds on a certain subspace. 
\item[(viii)] In \cite{IwaH05}, the authors generalize the relationship between (KYP) and (PR) for standard systems. Moreover, for descriptor systems our definition of (KYP) and (PR) can be seen as special case where $Q$ is replaced with $QE$. They show equivalence in the strict inequality case  between (KYP) and (PR) if $\det (s E -A) \neq 0$ for all $\re s \geq 0$, $E$ nonsingular and $Q$ positive definite.
\end{itemize}
 
In the remainder we will show that passivity implies positive realness, i.e.\ (Pa) implies (PR). 
For descriptor systems the state-control pair $(x,u)\in\K^{n}\times\K^m$ will in general not attain all values in $\K^{n}\times\K^m$ which is the reason why we cannot deduce (KYP) from (Pa) for descriptor systems.
Also positive realness together with certain controllability assumptions will only lead to solutions of the KYP inequality on a subspace. In the following we will use a more compact way of writing the KYP inequality for all $(x,u)\in\K^n\times\K^m$ as
\begin{align*}
&~~~~\re\begin{bmatrix}
Qx\\u
\end{bmatrix}^H\begin{bmatrix}
A&B\\-C&-D
\end{bmatrix}\begin{bmatrix}
x\\u
\end{bmatrix}\\&=\frac{1}{2}\begin{bmatrix}
Qx\\u
\end{bmatrix}^H\begin{bmatrix}
A&B\\-C&-D
\end{bmatrix}\begin{bmatrix}
x\\u
\end{bmatrix}+\frac{1}{2}\begin{bmatrix}
x\\u
\end{bmatrix}^H\begin{bmatrix}
A&B\\-C&-D
\end{bmatrix}^H\begin{bmatrix}
Qx\\u
\end{bmatrix}\\&=\frac{1}{2}\begin{bmatrix}
x\\u
\end{bmatrix}^H\begin{bmatrix}
Q^HA+A^HQ&&Q^HB-C^H\\B^HQ-C&&-D-D^H
\end{bmatrix}\begin{bmatrix}
x\\u
\end{bmatrix}\\&\leq 0
\end{align*}
Further, for Hermitian $A\in\K^{n\times n}$ and a subspace $\mathcal{V}$ of $\K^n$ the relation $A\geq_{\mathcal{V}}0$ means that $v^HAv\geq 0$ for all $v\in\mathcal{V}$. Then we consider the following restricted versions of \eqref{eq:kyp_dae}:
\begin{description}
\item[\rm (KYP$\mid\mathcal{V}$)] There exists $Q\in\K^{n\times n}$ with 
\begin{align}
\label{PH_on_V}
\re\begin{bmatrix}
Qx\\u
\end{bmatrix}^H\begin{bmatrix}
A&B\\-C&-D
\end{bmatrix}\begin{bmatrix}
x\\u
\end{bmatrix}\leq 0,\quad x^HE^HQx\geq 0,\quad \text{for all $\begin{bmatrix}
x\\ u
\end{bmatrix}\in\mathcal{V}$}.
\end{align}
\item[\rm (KYP$_{E}\mid\mathcal{V}$)] There exists $Y=Y^H\in\K^{n\times n}$ such that \eqref{PH_on_V} holds for some $Q=YE$. 
\end{description}

There are several authors who consider KYP inequalities restricted to suitable subspaces:
\begin{itemize}
\item[\rm (i)] In \cite{CamF09}, the authors  use $Ax$ instead of $\dot x$ in the KYP inequality \eqref{eq:kyp_dae} which leads to a $3\times 3$ block matrix where each block corresponds to one of the entries in $(\dot x,x,u)$. To show existence of KYP inequality solutions on a certain subspace for a positive real transfer function they assume minimality. Furthermore, the passivity notion is different since they allow for arbitrary nonnegative storage functions which must not vanish on $\ker E$. 
    
\item[\rm (ii)] The KYP inequality (KYP$_{E}$\textbar$\mathcal{V}$) is considered in  \cite{ReiS10} for the dual system restricted to the right deflating subspaces of the regular pair $(E,A)$ and show that the behavioral controllability and observability implies existence of solutions to this inequality. However, they use a passivity notion which is only equivalent to (Pa) for behaviorally controllable and observable systems.
\item[\rm (iii)] In \cite{ReiRV15,ReiV15}, the authors consider KYP inequalities restricted to the so called \emph{system space} $\Vs$, which is the smallest subspace where the system trajectories evolve, see \eqref{def:Vsys}. It is shown in that case that behavioral controllability and (PR) implies (KYP\textbar$\Vs$).
\item[\rm (iv)] The authors in \cite{ReiV19} consider the relation between storage functions, the feasibility of linear quadratic optimal control problems on an infinite time horizon and the existence of solution to KYP inequalities on the system space $\Vs$ which is defined in \eqref{def:Vsys}. However, the definition of storage function is slightly different from the notion used in (Pa) as e.g.\ the values of these functions might be negative.
\end{itemize}
Below, we focus on the approach presented in \cite{ReiRV15,ReiV15} because compared to \cite{ReiS10} the solution $Q$ might be nonzero on a larger subspace and the results in \cite{CamF09} require minimality to guarantee a solution to their KYP inequality.

In \cite{ReiV19} the authors consider solutions $(x,u)\in L^2_{loc}(\dR,\K^{n+m})$ of \eqref{DAE} in the space of Lebesgue measurable and locally square integrable functions where $Ex$ is differentiable almost everywhere with derivative $\tfrac{d}{dt}Ex\in L^2_{loc}(\dR,\K^{n})$ and \eqref{DAE} holds almost everywhere. Then the \emph{system space} $\Vs$ is defined as the smallest subspace of $\K^{n+m}$ such that for all solutions $(x,u)$ it holds that 
\begin{align}
\label{def:Vsys}
(x,u)\in L^2_{loc}(\dR,\Vs).
\end{align}
If $E$ is invertible then $\Vs=\K^n\times\K^m$ holds.
The system space was also introduced  in \cite{ReiRV15,ReiV15} using a slightly different solution concept. In this note, we rely on these results as well and therefore we provide a detailed comparison in Section \ref{sec:systemspace}. 

The restricted KYP inequalities (KYP\textbar$\Vs$) and  (KYP$_{E}$\textbar$\Vs$) were studied in \cite{ReiRV15,ReiV15} where the authors consider $E^HQ\geq_{\mathcal{V}_{\rm diff}} 0$, such that 
\[
\mathcal{V}_{\rm diff}:=\{x_0\in\dR^n ~\mid~  \text{$(x,u)\in L^2_{loc}(\dR,\K^{n+m})$ solves \eqref{DAE} with $Ex(0)=Ex_0$}\},
\]
instead of $E^HQ\geq_{\Vsx} 0$, where $\mathcal{V}^x$ is the projection of $\mathcal{V}$ to onto the first $n$ components. However, it is easy to see that $\mathcal{V}_{\rm diff}=\Vsx+\ker E$ and hence the properties used above are equivalent to the restricted KYP inequalities used in \cite{ReiRV15,ReiV15}.


We have the following results from \cite[Theorem 4.1, Proposition 4.4]{ReiRV15} and \cite[Theorem 4.3]{ReiV15} on the relation between (PR), (KYP$_{E}$\textbar$\Vs$) and (KYP\textbar$\Vs$).

\begin{proposition}
\label{prop:tima}
Let $\Sigma = (E,A,B,C,D)$ be a descriptor system of the form \eqref{DAE} and transfer function $\mathcal{T}(s)=C(sE-A)^{-1}B+D$. Then the following holds:
\begin{itemize}
    \item[\rm (a)] If there exists a solution $Y=Y^H$ to \emph{(KYP$_{E}$\textbar$\Vs$)} then $\mathcal{T}$ is positive real.  
    \item[\rm (b)] If the system is behaviorally controllable and $\mathcal{T}$ is positive real then there exists a solution $Y=Y^H$ to \emph{(KYP$_{E}$\textbar$\Vs$)}.
    \item[\rm (c)] If $Y$ satisfies \emph{(KYP$_{E}$\textbar$\Vs$)} then $Q:=EY$ satisfies \emph{(KYP\textbar$\Vs$)}.  
    \item[\rm (d)] If $Q$ satisfies \emph{(KYP\textbar$\Vs$)} then there exists a solution $Y$ to \emph{(KYP$_{E}$\textbar$\Vs$)} with $E^HYE=E^HQ$.  
\end{itemize}
\end{proposition}


In Proposition~\ref{prop:tima}~(b), we only consider the condition of behavioral controllability. The question that arises is whether the dual property, behavioral  observability, can be used as a condition for the existence of solutions to the KYP inequality. To answer this question, we consider the following example.

\begin{example}\label{ex:obsv!notcontr}
Consider $\dot x=x$, $y=cx+u$ for some $c\neq 0$. Then this system is observable but not controllable. Furthermore, $\mathcal{T}(s)=1$ for all $s\in\dC$ which is positive real. However, the KYP inequality which is given by
\[
\begin{bmatrix}
-2Q&& \overline{c}\\c&&2
\end{bmatrix}\geq 0
\]
is only solvable for sufficiently small $Q<0$. Hence, positive realness and observability is not enough to guarantee a solution to the KYP inequalities.
\end{example}

Example~\ref{ex:obsv!notcontr} shows that 
behavioral observability of $\Sigma=(E,A,B,C,D)$ together with a positive real transfer function does not imply (KYP$_{E}$\textbar$\Vs$) or (KYP\textbar$\Vs$). However, the behavioral observability of $\Sigma$ is equivalent to the behavioral controllability of the dual system $\Sigma'=(E^H,A^H,C^H,B^H,D^H)$ which leads us to the following result.
\begin{corollary}
Let $\Sigma=(E,A,B,C,D)$ be a positive real and behaviorally observable descriptor system. Then the dual system $\Sigma'=(E^H,A^H,C^H,B^H,D^H)$ is positive real and behaviorally controllable and there exists a solution to the generalized KYP inequality
\begin{align}
\label{eq:dual_kyp}
\begin{bmatrix}
-AQ-Q^HA^H&&B-Q^HC^H\\B^H-CQ&&D+D^H
\end{bmatrix}\geq 0,\quad EQ=Q^HE^H\geq 0.
\end{align}
\end{corollary}
\begin{proof}
The behavioral controllability of the dual system  holds by definition and \eqref{DAE_minimal}. Since for all $\lambda\in\dC$ with $\re\lambda\geq0$ it holds that
\[
B^H(\lambda E^H-A^H)^{-1}C^H+D^H+C(\overline{\lambda} E-A)^{-1}B+D=\mathcal{T}(\overline{\lambda})+\mathcal{T}(\overline{\lambda})^H\geq 0,
\]
the dual system is positive real and behaviorally controllable. Hence the existence of solutions to \eqref{eq:dual_kyp} follows form Proposition~\ref{prop:tima} (b),(c). 
\end{proof}

From Proposition~\ref{prop:tima} and \eqref{pass_ungl} we immediately obtain the following corollary. 
\begin{corollary}
\label{cor:restr}
Let $(E,A,B,C,D)$ be a descriptor system of the form \eqref{DAE}. Then the following holds:
\[
\emph{(Pa)}\quad \Longleftrightarrow\quad \emph{(KYP}\text{\textbar}\Vs) \quad \Longleftrightarrow \quad \emph{(KYP}_{E}\text{\textbar}\Vs)  \quad \Longrightarrow \quad \emph{(PR)}.
\]
Moreover, if the system is behaviorally controllable then \emph{(PR)} implies \emph{(Pa)}.
\end{corollary}
It was shown in \cite[Theorem 4.3]{ReiV15} that (PR) together with the conditions 
\begin{align}
\label{aspt:weaker}
\begin{split}
\rk[\lambda E-A,B]=n,\quad \text{for all  $\lambda\in\dC, \re\lambda\geq 0$},\\
\mathcal{T}(i\omega)+\mathcal{T}(i\omega)^H>0,\quad \text{for all $i\omega\in i\dR\setminus\sigma(E,A)$,}
\end{split}
\end{align}
implies (KYP\textbar$\Vs$) and hence (Pa). In the case of standard systems the controllability assumptions on a positive real system to fulfill (Pa) or equivalently (KYP), can be relaxed to the \emph{sign-controllability}, i.e.
\[
\rk[\lambda I_n-A,B]=n \quad \text{or}\quad \rk[-\overline{\lambda} I_n-A,B]=n\quad \text{for all  $\lambda\in\dC$},
\]
see \cite{Ferr05} and \cite[Section A.4]{BroLME07}.
In particular, the first condition in \eqref{aspt:weaker} implies the sign-controllability.
For descriptor systems \cite{ReiRV15} shows that \emph{behavioral sign-controllability} which means 
\[
\rk[\lambda E-A,B]=n \quad \text{or}\quad \rk[-\overline{\lambda} E-A,B]=n\quad \text{for all  $\lambda\in\dC$},
\]
together with a full rank assumption on the Popov function
\[
\rk \begin{bmatrix}
(-\overline{s}E-A)^{-1}B\\I_m
\end{bmatrix}^H\begin{bmatrix}
0&&C^H\\C&& \tfrac{1}{2}(D+D^H)
\end{bmatrix}\begin{bmatrix}
(sE-A)^{-1}B\\I_m
\end{bmatrix}=m
\]
where the rank is computed over the quotient field $\K(s)$, implies that there is a Hermitian solution $Q^HE=E^HQ$ to the generalized KYP inequality which is not necessarily nonnegative. The nonnegativity and hence (Pa) can be concluded from \cite[Theorem 4.3 (c)]{ReiV15} is the system is in addition \emph{behaviorally detectable} meaning that 
\[
\rk[\lambda E^\top-A^\top,C^\top]=n,\quad \text{for all $\lambda\in\dC$, $\re \lambda\geq 0$,}
\]
holds.

In addition to the presented results, the following example from \cite{BreS21} demonstrates that invertible solutions on the system space do not lead to a pH formulation. 
\begin{example}
\label{ex:tobiphil}
Let $E=\begin{bmatrix}
1&0\\0&0
\end{bmatrix}$, $A=\begin{bmatrix}
-1&0\\0&1
\end{bmatrix}$, $B=\begin{bmatrix}1\\0
\end{bmatrix}$, $C=\begin{bmatrix}1&1\end{bmatrix}$, $D=0$. Then this system is behaviorally observable,  behaviorally controllable, positive real with transfer function $\mathcal{T}(s)=\tfrac{1}{s+1}$ and the system space is $\Vs=\{(x_1,0,u)^\top\mid x_1,u\in\dR\}$. However, it is not minimal, since $A\ker E\nsubseteq \ran E$ and $\rk[E,B]=1$. Furthermore, the KYP inequality \eqref{eq:kyp_dae} has no solution but (KYP\textbar$\Vs$) holds since the system is behaviorally controllable. 
\end{example}


\section{When does (KYP) imply (pH)?}
\label{sec:KYP_pH}
In order to study if (KYP) imply (pH), we first need to consider KYP solutions for DAEs. Below we show that for DAEs with invertible $E$ the observability of the system guarantees invertible solutions of the KYP inequality. This result was already obtained in \cite{CamIV14} for positive semi-definite solutions of standard systems, but the proof can easily be extended to descriptor systems.
\begin{proposition}
\label{prop:obs}
Let $\Sigma=(E,A,B,C,D)$ be a descriptor system of the form \eqref{DAE} with $E$ invertible and $(E,A,C)$ behaviorally observable. If $Q\in\K^{n\times n}$ satisfies
\[
\begin{bmatrix}-A^HQ-Q^HA&& C^H-Q^HB\\C-B^HQ && D + D^H \end{bmatrix}\geq 0,\quad Q^HE=E^HQ,
\]
then $Q$ is invertible.
\end{proposition}
\begin{proof}
First, one shows that $\ker Q$ is an $(A,E)$-invariant subspace, i.e.\ $A\ker Q\subseteq E\ker Q$. Let $z=Av$ with $v\in\ker Q$ then 
\[
0=-v^HA^H Qv-v^H Q^HAv=v^H\underbrace{(-A^H Q- Q^HA)}_{\leq 0}v.
\]
Hence $v\in\ker(-A^H Q- Q^HA)$ which implies with $v\in\ker Q$ that $v\in\ker Q^HA$. Therefore $Q^Hz= Q^HAv=0$, i.e.\ $z\in\ker Q^H$. This shows $A\ker Q\subseteq\ker Q^H$. Moreover, since $E$ is invertible one has
\[
\ker Q^H=EE^{-1}\ker Q^H=E\ker Q^HE=E\ker (E^HQ)=E\ker Q.
\]
Hence $\ker Q$ is an $(A,E)$-invariant subspace. Next, we show that $\ker Q\subseteq\ker C$. Let $z\in\ker  Q$. Then the KYP inequality \eqref{eq:kyp_dae} implies for all $w\in\dR^m$ and $\alpha\in\dR$ that
\begin{align}
0&\leq \begin{bmatrix}
z\\ \alpha w
\end{bmatrix}^H\begin{bmatrix}-A^HQ-Q^HA&& C^H-Q^HB\\C-B^HQ && D + D^H \end{bmatrix}\begin{bmatrix}
z\\ \alpha w
\end{bmatrix}\nonumber\\&=z^HC^H\alpha w+\alpha w^HCz+\alpha^2w^H(D+D^H)w.\label{KYP_w_alph}
\end{align}
Assume that $Cz\neq 0$. Then all $\alpha>0$ sufficiently small would violate \eqref{KYP_w_alph}. This contradiction leads to $z\in\ker C$ which implies $\ker Q\subseteq\ker C$. Hence Proposition 7.2 in \cite{BerRT17} implies $\ker Q=\{0\}$. 
\end{proof}

In Proposition~\ref{prop:obs} we do not require the positive semi-definiteness of $Q^HE$ in (KYP). Hence as a special case we obtain the following result.
\begin{corollary}
Let $\Sigma=(E,A,B,C,D)$ be a descriptor system of the form \eqref{DAE} with $E$ invertible and $(E,A,C)$ behaviorally observable. If $Q\in\K^{n\times n}$ satisfies \emph{(KYP)} then $Q$ is invertible.
\end{corollary}

The Example \ref{ex:onlyzero} shows that the detectability of the system will in general not guarantee that the solutions of the KYP inequality are invertible.

Next, we extend the following result for standard systems \cite[p.~55]{Geop09} which indicates that a pH formulation of a passive system can be established without the observability assumption. 
\begin{proposition}
\label{prop:arjan}
Let $\Sigma=(E,A,B,C,D)$ define a descriptor system of the form \eqref{DAE}. Then \emph{(pH)} holds if and only if \emph{(KYP)} holds for some $Q\in\dC^{n\times n}$ with $\ker Q\subseteq\ker C\cap\ker A$.
\end{proposition}
\begin{proof}
If (pH) holds, then $A=(J-R)Q$ and $C=(G+P)^HQ$ for some $J,R,Q\in\K^{n\times n}$ and $G,P\in\K^{n\times m}$. Hence $\ker Q\subseteq\ker A\cap\ker C$ and by Proposition \ref{prop:dae_equiv}, $Q$ satisfies (KYP). To proof the sufficiency assume that (KYP) holds for some $Q$ with $\ker Q\subseteq\ker C\cap\ker A$. Then $Q^HE\geq 0$ holds and we can define $\Theta\in\K^{n\times n}$ via
\begin{align}
\label{def_sigma}
\begin{bmatrix}
A&B\\-C&-D
\end{bmatrix}=\Theta\begin{bmatrix}
Q&0\\0&I
\end{bmatrix}
\end{align}
which is well defined if $\ker Q\subseteq\ker A\cap\ker C$. Furthermore, (KYP) yields
\begin{align}
\label{dissip_with_Q}
\begin{bmatrix}
Q&0\\0&I
\end{bmatrix}^H(\Theta+\Theta^H)\begin{bmatrix}
Q&0\\0&I
\end{bmatrix}\leq 0.
\end{align}
If we set $\Gamma:=\tfrac{1}{2}(\Theta-\Theta^H)$ and $W:=\tfrac{1}{2}(\Theta+\Theta^H)$ then \eqref{def_PH} holds which finally proves (pH).
\end{proof}

\begin{remark}
If $Q$ is positive definite then \eqref{dissip_with_Q} is equivalent to $\Theta+\Theta^H\leq 0$ which also appears in some references as a definition of pH (descriptor) systems. 
If $Q$ is singular, we cannot not conclude $\Theta+\Theta^H\leq 0$ from \eqref{dissip_with_Q} in general. Instead we can redefine $\Theta$ in such a way that it fulfills \eqref{def_sigma}. To this end, we use the space decomposition $\K^n=\ran Q\oplus\ker Q^H$ and letting
\[
\Theta=\begin{bmatrix}
\Theta_1&\Theta_2\\\Theta_3&\Theta_4
\end{bmatrix}\in\K^{(n+m)\times(n+m)},\quad \Theta_1=\begin{bmatrix}
P_{\ran Q}\Theta_1\mid_{\ran Q}&P_{\ran Q}\Theta_1\mid_{\ker Q^H}\\P_{\ker Q^H}\Theta_1\mid_{\ran Q}&P_{\ker Q^H}\Theta_1\mid_{\ker Q^H}
\end{bmatrix},
\]
where $P_{\ran Q}$ is the orthogonal projector onto $\ran Q$. We can redefine $\Theta_1$ as 
\[
\hat \Theta_1:=\begin{bmatrix}
P_{\ran Q}\Theta_1\mid_{\ran Q}&-(P_{\ker Q^H}\Theta_1\mid_{\ran Q})^H\\P_{\ker Q^H}\Theta_1\mid_{\ran Q}&0
\end{bmatrix}
\]
without changing the product on the right hand side of \eqref{def_sigma}. Hence the matrix $\hat\Theta$ which is obtained after replacing the block $\Theta_1$ in $\Theta$ with $\hat \Theta_1$ fulfills \eqref{def_sigma} and $\hat\Theta+\hat\Theta^H\leq 0$. 
\end{remark}

\begin{remark}
The above result also holds for restrictions on subspaces, e.g.\ the system space $\Vs$. Furthermore, if a given descriptor system $\Sigma=(E,A,B,C,D)$ is behaviorally observable then the matrix $Q$ in (pH) fulfills $\ker Q\subseteq\ker A\cap\ker C=\{0\}$ by  \cite[Proposition 7.2, Theorem 7.3]{BerRT17} and hence $Q$ must be invertible. 
\end{remark}
\begin{remark}
Another way to use solutions $Q$ to the KYP inequalities to obtain a representation which is quite similar to a pH formulation and does not require further assumptions is given by a left-multiplication of the state equations with $Q^H$
\begin{align*}
\hat\Theta=\begin{bmatrix}
J-R&&G-P\\-(G+P)^H&&-S-N
\end{bmatrix}:=\begin{bmatrix}
Q^H&0\\0&I
\end{bmatrix}\begin{bmatrix}
A&B\\-C&-D
\end{bmatrix}.
\end{align*}
However, if $Q$ is not invertible then this multiplication might enlarge the solution set of the descriptor system. The treatment of pH systems with singular $Q$ is  described in detail in \cite[Section 6.3]{MehMW20}, see also \cite{UnM22}.
\end{remark}

\section{When does (Pa) imply (pH)?}
\label{sec:Pa_pH}
In the previous section we have discussed when (KYP) implies (pH). Although all solutions to the KYP inequality \eqref{eq:kyp_dae} fulfill $\ker Q\subseteq\ker C$ which follows from the proof of Proposition~\ref{prop:obs}, there will in general not exist a solution whose kernel is a subset of $\ker A\cap\ker C$ as Example~\ref{ex:onlyzero} shows. Hence a pH formulation of a given passive system is not always possible without further assumptions. 

If $E$ is invertible and (Pa) holds then Proposition~\ref{prop:dae_equiv} yields (KYP) and therefore the results of the previous section can be used to characterize when (pH) holds. Hence we will focus in this section on the case where $E$ is not invertible. 

As a first result, we show that for passive systems one might restrict to systems which are controllable and observable with index at most two. First, we recall the notion of index of descriptor systems \eqref{DAE}. Since $(E,A)$ is assumed to be regular, there exist invertible $S,T\in\dC^{n\times n}$ and $r\in\mathbb{N}$, see e.g.\ , such that 
\begin{align}
\label{eq:wcf}
(SET,SAT)=\left(\begin{bmatrix}I_{r}&0\\0&N\end{bmatrix},\begin{bmatrix}J&0\\0&I_{n-r}\end{bmatrix}\right)
\end{align}
where $J$ and $N$ are in Jordan canonical form and $N$ is nilpotent. The block-diagonal form \eqref{eq:wcf} is typically referred to as Kronecker-Weierstra\ss\ form. Based on this form, the \emph{index} of the system \eqref{DAE} is defined as the smallest natural number $\nu$ such that $N^{\nu-1}\neq0$ and $N^{\nu}=0$.

If a given system in state-space form is not controllable or observable then one might consider the KYP inequality  restricted to observable or controllable subspaces. This was proposed for standard systems in \cite{SchW94,XiaH99}, see also \cite[Theorem 3.39]{BroLME07}.

For descriptor systems we recall a Kalman-like decomposition from \cite[Theorem 8.1]{BerRT17}, see also \cite{BanKL92} and  \cite[p.\ 51ff]{Dai89}. 
\begin{proposition}
\label{prop:kalman}
For $E,A\in\K^{l\times n}$, $B\in\K^{n\times m}$ and $C\in\K^{m\times n}$ there exists invertible $S\in\K^{l\times l}$, $T\in\K^{n\times n}$ such that 
\begin{multline*}
[SET,SAT,SB,CT]=\\
\left[\begin{bmatrix}
E_{11}&E_{12}&E_{13}&E_{14}\\0&E_{22}&0&E_{24}\\0&0&E_{33}&E_{34}\\ 0&0&0&E_{44}
\end{bmatrix},\begin{bmatrix}
A_{11}&A_{12}&A_{13}&A_{14}\\0&A_{22}&0&A_{24}\\0&0&A_{33}&A_{34}\\ 0&0&0&A_{44}
\end{bmatrix},\begin{bmatrix}
B_1\\B_2\\0\\0
\end{bmatrix},[0~C_2~0~C_4]\right].
\end{multline*}
Furthermore, the subsystems
\[
\left[\begin{bmatrix}
E_{11}&E_{21}\\0&E_{22}
\end{bmatrix},\begin{bmatrix}
A_{11}&A_{21}\\0&A_{22}
\end{bmatrix},\begin{bmatrix}
B_{1}\\B_{2}
\end{bmatrix}\right],\quad \left[\begin{bmatrix}
E_{22}&E_{24}\\0&E_{44}
\end{bmatrix},\begin{bmatrix}
A_{22}&A_{24}\\0&A_{44}
\end{bmatrix},\begin{bmatrix}
C_{2} & C_{4}
\end{bmatrix}\right]
\]
are completely controllable and observable, respectively.
\end{proposition}

Since all properties are invariant under pencil equivalence the transfer function equals
\begin{align*}
\mathcal{T}(s)=C(sE-A)^{-1}B+D&=CT(sSET-SAT)^{-1}SB+D\\&=C_2(sE_{22}-A_{22})^{-1}B_2+D.
\end{align*}
If the system is passive then by Corollary~\ref{cor:restr}, $\mathcal{T}(s)$ is positive real. Hence $C_2(sE_{22}-A_{22})^{-1}B_2+D$ is positive real. We want to conclude now that the index of $(E_{22},A_{22})$ is at most two. The complete controllability and observability of the subsystems is equivalent to
\[
\rk[E_{22},B_2]=\rk[E_{22}^H,C_2^H]=n.
\]
These conditions are invariant under the transformation used in Proposition~\ref{prop:kalman} and hence we can assume without restriction that $E_{22}$ is a Jordan block at $0$ and $A_{22}$ is the identity. Using the nilpotency we find that $(sE_{22}-A_{22})^{-1}=-\sum_{i=0}^{\nu-1}(sE_{22})^i$ where $\nu$ is the nilpotency index of $E_{22}$. The rank condition implies $C_2E_{22}^{\nu-1}B_2\neq 0$. The positive realness implies that $\nu\leq 2$. 

Hence, one way to obtain a pH formulation is to restrict to the controllable and observable subspace, i.e.\ one considers the completely controllable and observable subsystem $(E_{22},A_{22},B_2,C_2,D)$ which has index at most two. 

Numerically, it is beneficial to use only unitary or orthogonal transformations $S$ and $T$ to obtain the Kalman-like form given in Proposition \ref{prop:kalman} as discussed in \cite{BunBMN99}, see also \cite[Section 7]{UnM22}.

\subsection{From (KYP\textbar$\Vs$) to (KYP)}
It was shown in Corollary~\ref{cor:restr} that passive systems have a solution $Q\in\K^{n\times n}$ to (KYP\textbar$\Vs$). In this section, we will rewrite this restricted KYP inequality and derive an equivalent KYP inequality which holds on the whole space. The idea is to choose a basis of $\Vs$ and redefine the system matrices in such a way that the system space of the new system will be the whole space.

Let $M_{\Vs}\in\K^{n\times\dim\Vs}$ and and $M_{\Vsh}\in\K^{n\times\dim\Vs}$, $\Vsh=\begin{smallbmatrix}A&B\\-C&-D
\end{smallbmatrix}\Vs$  be matrices whose columns are a basis of $\Vs$ and a span of $\Vsh$, respectively.
Using the Moore-Penrose pseudo-inverse $M^\dagger$ of $M\in\K^{k\times l}$, see e.g.\ \cite[Section 5.5.2]{GoluL13}, we introduce the system 
\[
\begin{bmatrix}
A_{sys}&B_{sys}\\
-C_{sys}&-D_{sys}
\end{bmatrix}:=M_{\Vsh}^\dagger\begin{bmatrix}
A&B\\-C&-D
\end{bmatrix}M_{\Vs}.
\]
If we consider the KYP inequality together with the projector formulas based on the pseudo-inverse
\[
P_{\Vsh}=M_{\Vsh}M_{\Vsh}^\dagger,
\]
then for all $(x,u) \in \Vs$ with $w \in \K^{\dim\Vs}$ such that $(x,u)= M_{\Vs} w$ we obtain
\begin{align*}
0&\geq \re \begin{bmatrix}
x\\u
\end{bmatrix}^H\begin{bmatrix}
Q^H&0\\0&I_m
\end{bmatrix}\begin{bmatrix}
A&B\\-C&-D
\end{bmatrix}\begin{bmatrix}
x\\u
\end{bmatrix}\\
&=\re (M_{\Vs} w)^H\begin{bmatrix}
Q^H&0\\0&I_m
\end{bmatrix}M_{\Vsh}M_{\Vsh}^\dagger \begin{bmatrix}
A&B\\-C&-D
\end{bmatrix}M_{\Vs} w,
\end{align*}
and using $[A,B]\Vs\subseteq [E,0]\Vs$ yields
\begin{align*}
x^HQ^HEx&=\begin{bmatrix}
x\\u
\end{bmatrix}^H\begin{bmatrix}
Q^H&0\\0&I_m
\end{bmatrix}\begin{bmatrix}E&0\\0&0\end{bmatrix}\begin{bmatrix}
x\\u
\end{bmatrix}\\&=(M_{\Vs} w)^H\begin{bmatrix}
Q^H&0\\0&I_m
\end{bmatrix}\begin{bmatrix}E&0\\0&0\end{bmatrix}M_{\Vs} w\\
&=(M_{\Vs} w)^H\begin{bmatrix}
Q^H&0\\0&I_m
\end{bmatrix}M_{\Vsh}M_{\Vsh}^\dagger\begin{bmatrix}E&0\\0&0\end{bmatrix}M_{\Vs} w.
\end{align*}

In a more simple way, this can be rewritten with 
\[
\hat Q^H:=M_{\Vs}^H\begin{bmatrix}
Q^H&0\\0&I_m
\end{bmatrix}M_{\Vsh},\quad \hat E:=M_{\Vsh}^\dagger\begin{bmatrix}E&0\\0&0\end{bmatrix}M_{\Vs}
\]
as 
\begin{align}
\label{eq:KYP_nondiag}
\re w^H \hat Q^H\begin{bmatrix}
A_{sys}&B_{sys}\\-C_{sys}&-D_{sys}\end{bmatrix} w\leq 0,\quad  w\in\K^{\dim\Vs},\quad \hat Q^H\hat E\geq 0.
\end{align}
This is in fact the KYP inequality of a standard system which is obtained by using $\ran E$ as state space and treating the remaining variables in $\K^{\dim \Vs}$ as input variables. 
However, the matrix $\hat Q$ which replaces $\begin{smallbmatrix}
Q&0\\0&I_m
\end{smallbmatrix}$ in the KYP inequality \eqref{eq:kyp_dae} of the standard system must in general not admit this diagonal structure. This seems reasonable because the controls and states are not decoupled for general descriptor systems. For pH descriptor systems the interplay between state and control variables can be seen from a staircase form which was derived in \cite{BeaGM19} and can be achieved using only unitary transformations.

Analogously to Proposition~\ref{prop:arjan}, a pH formulation of the system $\Sigma_{sys}$ is given if 
\[
\ker \hat Q\subseteq\ker \begin{bmatrix}
A_{sys}&B_{sys}\\
-C_{sys}&-D_{sys}
\end{bmatrix}
\]
holds. With this assumption we can define a matrix $\hat \Theta\in\K^{\dim\Vs\times\dim\Vs}$ as in Proposition~\ref{prop:arjan} by setting
\[
\begin{bmatrix}
A_{sys}&B_{sys}\\
-C_{sys}&-D_{sys}
\end{bmatrix}=\hat \Theta\hat Q.
\]

Motivated by the modified KYP inequality \eqref{eq:KYP_nondiag} we enlarge the class of solutions of the KYP \eqref{eq:kyp_dae} and show how in this case a pH formulation can be obtained. In the following we will restrict to standard systems  $\Sigma=(I_n,A,B,C,D)$ and study solutions $\hat Q\in\K^{(n+m)\times(n+m)}$
\begin{align}
    \label{ineq:Lyap_Q}
\re \begin{bmatrix}x\\u
\end{bmatrix}^H\hat Q\begin{bmatrix}
A&B\\-C&-D\end{bmatrix}\begin{bmatrix}x\\u\end{bmatrix}\leq 0,\quad \hat Q=\hat Q^H\geq 0 \quad \text{for all } (x,u)\in \K^n \times \K^m.
\end{align}
Although every ordinary pH system fulfills \eqref{ineq:Lyap_Q}, the converse is not necessarily true as the following example shows. 
\begin{example}
We consider
\[
\dot x=x+3u,\quad y=3x+5u.
\]
This system does not fulfill (pH) since it is unstable. However the matrix   $\begin{smallbmatrix}
A&B\\-C&-D
\end{smallbmatrix}=\begin{smallbmatrix}
1&3\\-3&-5
\end{smallbmatrix}$ used in \eqref{ineq:Lyap_Q} has the double eigenvalue $\lambda_{1,2}=-2$. Hence the Lyapunov inequality~\eqref{ineq:Lyap_Q} has a positive definite solution $\hat Q\in\dR^{2\times 2}$.
\end{example}

In the following, we show different ways of obtaining a pH representation of the system if we additionally assume that $\hat Q$ in \eqref{ineq:Lyap_Q} is positive definite. The first way is to define
\[
\begin{bmatrix}
J-R&&B-P\\-B^H-P^H&&-N-S
\end{bmatrix}:=\begin{bmatrix}
A&B\\-C&-D
\end{bmatrix}\hat Q^{-1}.
\]
Since $\hat Q$ is not block diagonal, the new state and input variables might be given as a combination of the old state and input variables. Another way to obtain a pH formulation where either the new state variable or the new input variable can be chosen as a rescaling of the old variables is based on the following block Cholesky factorization, see e.g.\ \cite[Section 4.2.9]{GoluL13},
\begin{align}
\label{eq:Q_chol}
\hat Q=\mathcal{C}^H\mathcal{C}=\begin{bmatrix}
C_1&C_2\\0&C_3
\end{bmatrix}^H\begin{bmatrix}
C_1&C_2\\0&C_3
\end{bmatrix}.
\end{align}
This can be used to rewrite \eqref{ineq:Lyap_Q} as follows
\begin{align}
\re \left(\mathcal{C}\begin{bmatrix}x\\u
\end{bmatrix}\right)^H\left(\mathcal{C}\begin{bmatrix}
A&B\\-C&-D\end{bmatrix}\mathcal{C}^{-1}\right)\mathcal{C}\begin{bmatrix}x\\u\end{bmatrix}=\re \begin{bmatrix}x\\u
\end{bmatrix}^H\mathcal{C}^H\mathcal{C}\begin{bmatrix}
A&B\\-C&-D\end{bmatrix}\begin{bmatrix}x\\u\end{bmatrix}\leq 0.
\end{align}
Hence we could rewrite the system using the new variables
\[
\begin{bmatrix}
\hat x\\ \hat u
\end{bmatrix}=\mathcal{C}\begin{bmatrix}x\\u
\end{bmatrix}=\begin{bmatrix}
C_1x+C_2u\\C_3u
\end{bmatrix}\quad \text{and}\quad \begin{bmatrix}
\hat A&\hat B\\-\hat C&-\hat D\end{bmatrix}:=\mathcal{C}\begin{bmatrix}
A&B\\-C&-D\end{bmatrix}\mathcal{C}^{-1}.
\]
We obtain a pH system with Hamiltonian $\hat Q=I_n$, state $\hat x$ and input $\hat u$. If we would choose $\mathcal{C}$ to be lower triangular in the Cholesky factorization \eqref{eq:Q_chol} then the new state $\hat x$ would be a rescaling of the old state $x$ whereas $\hat u$ would be a linear combination of the state and input variables $x$ and $u$.

Using the particular structure of $\mathcal{C}$ we further obtain
\begin{align*}
&~~~\begin{bmatrix}
\hat A&\hat B\\-\hat C&-\hat D\end{bmatrix}=\mathcal{C}\begin{bmatrix}
A&B\\-C&-D\end{bmatrix}\mathcal{C}^{-1}\\&=\begin{bmatrix}
C_1A-C_2C&&C_1B-C_2D\\-C_3C&&-C_3D\end{bmatrix}\mathcal{C}^{-1}\\
&=\begin{bmatrix}
C_1A-C_2C&&C_1B-C_2D\\-C_3C&&-C_3D\end{bmatrix}\begin{bmatrix} C_1^{-1}&& -C_1^{-1}C_2C_3^{-1}\\0&&C_3^{-1}\end{bmatrix}\\
&=\begin{bmatrix}
C_1AC_1^{-1}-C_2CC_1^{-1}&&C_1BC_3^{-1}-C_2DC_3^{-1}-C_1AC_1^{-1}C_2C_3^{-1}+C_2CC_1^{-1}C_2C_3^{-1}\\-C_3CC_1^{-1}&&-C_3DC_3^{-1}+C_3CC_1^{-1}C_2C_3^{-1}\end{bmatrix}
\end{align*}
Furthermore, the matrix $\hat A$ satisfies $\hat A+\hat A^H\leq 0$ and hence it is \emph{stable}, i.e.\ it has only eigenvalues with nonpositive real part and semi-simple eigenvalues on the imaginary axis (if there are any). 

\subsection{Passive systems with index at most one}
As we have seen in the previous subsection, passive systems have always a realization which has index at most two and that the index two blocks can be realized separately as pH systems. Hence the question remains whether passive systems with index at most one can be realized as pH systems.

In the following, we construct a pH realization for descriptor systems $\Sigma=(E,A,B,C,D)$ which are behaviorally observable and have index at most one. 

If $\Sigma$ has index at most one then, similar to the Kronecker-Weierstra\ss\ form \eqref{eq:wcf}, there exist invertible $T_l,T_r\in\dC^{n\times n}$ such that 
\begin{align}
\label{eq:semi-expl}
T_lET_r=\begin{bmatrix}
E_1&0\\0&0
\end{bmatrix},\quad T_lAT_r=\begin{bmatrix}
A_1&0\\0&A_2
\end{bmatrix},\quad T_lB=\begin{bmatrix}
B_1\\B_2
\end{bmatrix},\quad CT_r=[C_1,C_2]
\end{align}
with $n = n_1 + n_2$, $E_1\in\K^{n_1\times n_1}$ and $A_2\in\K^{n_2\times n_2}$ invertible. Based on this transformation, we consider the following KYP inequality
\begin{align}
\label{KYP_ind_1_simple}
\begin{split}
\begin{bmatrix}
-A_1^HQ_1-Q_1^HA_1& C_1^H-Q_1^HB_1
\\ C_1-B_1^HQ_1&(D-C_2A_2^{-1}B_2)+(D-C_2A_2^{-1}B_2)^H
\end{bmatrix}&\geq 0,\quad  x_1\in\K^{n_1}\\
\quad 
x_1^H
E_1^HQ_1x_1&\geq 0. 
\end{split}
\end{align}

\begin{proposition}
\label{prop:indexone}
Let $\Sigma=(E,A,B,C,D)$ be a descriptor system with index at most one which satisfies \eqref{eq:semi-expl}. Then the modified KYP inequality \eqref{KYP_ind_1_simple} has a solution $Q_1$ if and only if $\Sigma$ is passive. 
Moreover, if the system is behaviorally observable then every solution $Q_1$ is invertible.
\end{proposition}
\begin{proof}
We apply Corollary~\ref{cor:restr} where we showed that passivity is equivalent to (KYP\textbar$\Vs$). The system space of the block diagonal system \eqref{eq:semi-expl} is given by 
\[
\{(x_1,-A_2^{-1}B_2u,u) \mid x_1\in\K^{n_1}, u\in\K^m\}
\]
and therefore
\begin{align}
\label{Vsys_ind_one}
\Vs=\di(T_r,I_m)\{(x_1,-A_2^{-1}B_2u,u) \mid x_1\in\K^{n_1}, u\in\K^m\}.
\end{align}
Hence the KYP inequalities restricted to the system space are given by 
\[
\re \begin{bmatrix}
x\\u
\end{bmatrix}^H\begin{bmatrix}
Q&0\\0&I_m
\end{bmatrix}^H\begin{bmatrix}
A&B\\-C&-D
\end{bmatrix}\begin{bmatrix}
x\\u
\end{bmatrix}\leq 0,\quad x^HE^HQx\geq 0,\quad \text{for all $\begin{bmatrix}
x\\u
\end{bmatrix}\in\Vs$},
\]
which is equivalent to
\begin{align*}
\re \begin{bmatrix}
x_1\\-A_2^{-1}B_2u\\u
\end{bmatrix}^H \begin{bmatrix}
T_l^{-H}QT_r&0\\0&I_m
\end{bmatrix}^H \begin{bmatrix}
A_1&0&B_1\\ 0&A_2&B_2\\-C_1&-C_2&-D
\end{bmatrix}\begin{bmatrix}
x_1\\-A_2^{-1}B_2u\\u
\end{bmatrix}\leq 0,\\
\quad \begin{bmatrix}
x_1\\-A_2^{-1}B_2u
\end{bmatrix}^H\begin{bmatrix}
E_1^H&0\\0&0
\end{bmatrix}T_l^{-H}QT_r\begin{bmatrix}
x_1\\-A_2^{-1}B_2u
\end{bmatrix}\geq 0\quad
x_1\in\K^{n_1},~u\in\K^m.
\end{align*}
Since $E^HQ=Q^HE$, the following matrix is Hermitian
\[
\begin{bmatrix}
E_1^H&0\\0&0
\end{bmatrix}T_l^{-H}QT_r=\begin{bmatrix}
E_1^H&0\\0&0
\end{bmatrix}\begin{bmatrix}
Q_1&Q_2\\ Q_3& Q_4
\end{bmatrix}=\begin{bmatrix}
E_1^HQ_1&E_1^H Q_2\\0&0
\end{bmatrix}=\begin{bmatrix}
E_1^H Q_1&0\\0&0
\end{bmatrix}.
\]
The invertibility of $E_1$ implies $ Q_2=0$, i.e.\ $T_l^{-H}QT_r$ is block lower-triangular. Hence the KYP inequalities on the system space are equivalent to
\begin{align*}
\re \begin{bmatrix}
x_1\\u
\end{bmatrix}^H\begin{bmatrix}
Q_1&0\\0&I_m
\end{bmatrix}^H \begin{bmatrix}
A_1&B_1\\-C_1&C_2A_2^{-1}B_2-D
\end{bmatrix}\begin{bmatrix}
x_1\\u
\end{bmatrix}&\leq 0,\quad  x_1\in\dR^{n_1},~u\in\K^m\\ 
x_1^H
E_1^HQ_1x_1&\geq 0.
\end{align*}
If the system $\Sigma$ is behaviorally observable, then $(E_1,A_1,C_1)$ is observable and hence, using Proposition~\ref{prop:obs}, we find that $Q_1$ is invertible. 
\end{proof}
Proposition~\ref{prop:indexone} can be used to obtain a pH representation of $(E_1,A_1,B_1,C_1,D-C_2A_2^{-1}B_2)$ as follows. If there exists a solution $Q_1$ to \eqref{KYP_ind_1_simple} with $\ker Q_1\subseteq\ker A_1\cap\ker C_1$ then
a pH representation can be obtained 
by considering $\Theta$ given by
\[
\begin{bmatrix}
A_1&&B_1\\-C_1&& C_2A_2^{-1}B_2-D
\end{bmatrix}=\Theta\begin{bmatrix}
Q_1&0\\0&I_m
\end{bmatrix},
\]
and if $Q_1$ is invertible then this simplifies to
\[
\Theta=\begin{bmatrix}
A_1Q_1^{-1}&&B_1\\-C_1Q_1^{-1}&& C_2A_2^{-1}B_2-D
\end{bmatrix}.
\]

\subsection{Passivity and available storage}
 In his seminal paper \cite{Wil72} Willems observed that the passivity of standard systems can also be characterized in terms of the finiteness of the \emph{available storage} which can be interpreted as the maximal amount of energy that one might  extract from the system. For descriptor systems this quantity can be introduced as
\begin{align}
\label{def_Va}
\mathcal{S}_a(x_0):=\sup_{\substack{u\in \mathcal{C}^{\infty}([0,T],\K^m)\\ T\geq 0}}-\int_0^T \re y(\tau)^H u(\tau)d\tau
\end{align}
where we consider only consistent initial values $x_0$ and allow in the supremum in \eqref{def_Va} only inputs $u$ which are from  the space of smooth functions  $\mathcal{C}^{\infty}([0,T],\K^m)$. By definition the above number is real-valued and nonnegative since we can choose $T=0$ in \eqref{def_Va}. In the following we will extend the function  $V_a$ to $\K^n$ by using the orthogonal decomposition  $\K^n=\mathcal{V}_c\oplus\mathcal{V}_c^\perp$ where  $\mathcal{V}_c$ denotes the space of consistent initial values. Then $x\in\K^n$ can be uniquely decomposed as $x=x_0\oplus x_1$ for some $x_0\in\mathcal{V}_c$ and $x_2\in\mathcal{V}_c^\perp$ and we define $\mathcal{S}_a(x):=\mathcal{S}_a(x_0)$.

The aim of this section is to show that (Pa) and a finite available storage $\mathcal{S}_a$ are equivalent for descriptor systems without further controllability or observability assumptions. In this case the available storage $\mathcal{S}_a$ is a minimal storage function in the sense that $\mathcal{S}_a(x_0)\leq \mathcal{S}(x_0)$ for any storage function $\mathcal{S}$. The existence of $\mathcal{S}_a(x_0)$ is also interesting from an optimal control perspective since the resulting problem is singular indefinite. These problems were treated for standard systems in \cite{Gee89} and for descriptor systems in \cite{ReiV19}. Furthermore, the minimization of the supply rate was considered for pH systems in \cite{FauM21}.

The notion of available storage and its counter part, the \emph{required supply}, was extended to minimal descriptor systems in \cite{ReiS10}. However, the authors considered a different notion of passivity and assume instead of (Pa)
\begin{align}
\label{pass_0}
\int_0^T y(\tau)^\top u(\tau)d\tau\geq 0,\quad x(0)=0,
\end{align}
for all $T\geq 0$ and all $u\in L^2([0,T],\dR^m)$ for which a solution to \eqref{DAE} exists. The space $L^2([0,T],\K^m)$ consists of Lebesgue measurable functions $f: [0,T] \to \K^m$ with $\int_0^T \lvert f(\tau) \lvert^2 d\tau < \infty$. In comparison to (Pa), we consider in \eqref{pass_0} only the initial value $x_0=0$. Therefore, if (Pa) holds for some storage function  $\mathcal{S}$ then $\mathcal{S}(0)=0$ and hence (Pa) implies \eqref{pass_0}. Willems showed in \cite[Theorem 1]{Wil71} that for controllable standard systems that \eqref{pass_0} is equivalent to (Pa). Furthermore, if we assume \eqref{pass_0} for all initial values $x_0\in\dR^n$, i.e.\ when we replace $x(0)$ in \eqref{pass_0} by $x(0)=x_0$ for all $x_0\in\K^n$, then (Pa) holds for $\mathcal{S}=0$. For uncontrollable systems, \eqref{pass_0} does not imply passivity, see Example \ref{ex:obsv!notcontr}. 

By following the steps in the proof of \cite[Theorem 3.11]{ReiV19} we obtain that $\mathcal{S}_a$ is a nonnegative quadratic function.
\begin{lemma}
\label{lem:Va}
Let $\Sigma=(E,A,B,C,D)$ be a descriptor system for which $\mathcal{S}_a(\cdot)$ is finite. Then there exists $S_a\in\K^{n\times n}$ with $S_a=S_a^H\geq 0$ satisfying $\mathcal{S}_a(x)=x^HS_ax$ for all $x\in\K^n$.
\end{lemma}
\begin{proof}
We will abbreviate for $t_1\geq 0$
\[
\mathcal{J}_{t_1}(u,x,x_0):=-\frac{1}{2}\int_0^{t_1}\begin{bmatrix}
x(t)\\ u(t)
\end{bmatrix}^H\begin{bmatrix}
0&C^H\\C& D+D^H
\end{bmatrix}\begin{bmatrix}
x(t)\\ u(t)
\end{bmatrix}d\tau=-\int_0^{t_1}\re y(\tau)^Hu(\tau)d\tau
\]

We will show that $\mathcal{S}_a(\cdot)$ is a quadratic form. To this end, we have to show that
\begin{align}
\mathcal{S}_a(\lambda x_0)&=\text{\textbar}\lambda\text{\textbar}^2\mathcal{S}_a(x_0),\quad \text{for all $\lambda\in\dC$,}\label{show_scaling}\\
\mathcal{S}_a(x_{1}-x_{2})+\mathcal{S}_a(x_{1}+x_{2})&=2\mathcal{S}_a(x_{1})+2\mathcal{S}_a(x_{2}),\quad \text{for all $x_{1},x_{2}\in\K^n$}.\label{show_parallel}
\end{align}
As a consequence, there exists $S_a=S_a^H\geq 0$ satisfying $\mathcal{S}_a(x)=x^HS_ax$ for all $x\in\K^n$ which would prove the lemma.

To prove \eqref{show_scaling}, we first show $\mathcal{S}_a(0)=0$. Otherwise,  $\mathcal{S}_a(0)>0$ holds. Then by definition of the supremum, there exist $t_1>0$ and functions $x:[0,t_1]\rightarrow\K^n$, $x(0)=0$ and $u:[0,t_1]\rightarrow\K^m$ such that
\[
-\int_0^{t_1}\re y(\tau)^Hu(\tau)d\tau>0.
\]
Since for all $\alpha\geq 0$, $\alpha x(\cdot)$ and $\alpha u(\cdot)$ are solutions to \eqref{DAE} we have for the supremum 
\[
\infty>\mathcal{S}_a(0)\geq \sup_{\alpha\geq 0}-\int_0^{t_1}\re \alpha y(\tau)^H\alpha u(\tau)d\tau=\sup_{\alpha\geq 0}-\alpha^2\int_0^{t_1}\re y(\tau)^Hu(\tau)d\tau=\infty
\]
which is the desired contradiction. Consider now \eqref{show_scaling} for $\lambda\neq 0$. Then for fixed $t_1>0$ and $\hat u\in \mathcal{C}^{\infty}([0,t_1],\K^m)$
\begin{align}
\label{without_t1}
\mathcal{J}_{t_1}(\hat u,x,\lambda x_0)\leq \text{\textbar}\lambda\text{\textbar}^2 \sup_{\substack{u\in \mathcal{C}^{\infty}([0,t_1],\K^m)}}\mathcal{J}_{t_1}(u,x,x_0),
\end{align}
where we use in the estimate that $\lambda x$ and $\lambda u$ solve \eqref{DAE} with $\lambda x(0)=\lambda x_0$ if and only if $x$ and $u$ solve \eqref{DAE} for $x(0)=x_0$. Considering the supremum of \eqref{without_t1} for all $t_1\geq 0$ and $\hat u\in \mathcal{C}^{\infty}([0,t_1],\K^m)$ we obtain
\[
\mathcal{S}_a(\lambda x_0)\leq \text{\textbar}\lambda\text{\textbar}^2\mathcal{S}_a(x_0)
\]
for all $\lambda\neq 0$ and $x_0\in\K^n$. Furthermore, by the definition of the supremum there exists for all $k\geq 1$ functions $x_k:[0,t_1]\rightarrow\K^n$ and $u_k:[0,t_1]\rightarrow\K^m$ satisfying \eqref{DAE} and 
\begin{align*}
\text{\textbar}\lambda\text{\textbar}^2\sup_{\substack{u\in \mathcal{C}^{\infty}([0,t_1],\K^m)}}\mathcal{J}_{t_1}(u,x,x_0)-\text{\textbar}\lambda\text{\textbar}^2k^{-1}&\leq \text{\textbar}\lambda\text{\textbar}^2\mathcal{J}_{t_1}(u_k,x_k,x_0)\\&=\mathcal{J}_{t_1}(\lambda u_k,\lambda x_k,\lambda x_0)\\&\leq \sup_{\substack{u\in \mathcal{C}^{\infty}([0,t_1],\K^m)}}\mathcal{J}_{t_1}(u,x,\lambda x_0).
\end{align*}
Since this holds for all $k\geq 1$ it also holds for the limit $k\rightarrow\infty$. This together with the supremum for all $t_1\geq 0$ implies
\[
\mathcal{S}_a(\lambda x_0)\geq \text{\textbar}\lambda\text{\textbar}^2\mathcal{S}_a(x_0).
\]
The proof of \eqref{show_parallel} is analogous to the proof of \eqref{show_scaling}. 
\end{proof}

The following example shows that storage functions depending on $Ex$ instead of $x$ cannot be used to conclude passivity of the system. The idea here is that we can hide a negative feedthrough term in the algebraic part of the descriptor system. 
\begin{example}
Consider the real scalar system $0=x+u$, $y=x$, then $y=-u$ and one would expect that it is not passive. Indeed, for every nonnegative storage function $\mathcal{S}:\dR\rightarrow[0,\infty)$ we have
\[
\int_0^T y(\tau)^\top u(\tau)d\tau=-\int_0^T\text{\textbar}u(\tau)\text{\textbar}^2d\tau=\mathcal{S}(x(T))\geq 0
\]
which cannot hold for all inputs $u\in \mathcal{C}^{\infty}([0,T],\dR^m)$ and $T\geq 0$. Furthermore, $E=0$ and therefore $Ex_0=0$ for all $x_0\in\dR$ and, hence, the value of all storage functions depending on $Ex_0$ would be zero and in particular finite. Hence a storage function based on $Ex_0$ cannot be used to characterize the passivity. 
\end{example}

It is well-known for standard passive systems that the available storage $\mathcal{S}_a$ is the minimal storage function, see \cite{Wil72}, in the sense that $\mathcal{S}_a(x)\leq \mathcal{S}(x)$ holds for all $x\in\K^n$ and all storage functions $\mathcal{S}$. This was extended in \cite{ReiS10} to descriptor systems satisfying certain minimality assumptions. In the following proposition we show this property without assuming any controllability or observability of the system. 
\begin{proposition}
\label{prop: finit}
Let $\Sigma=(E,A,B,C,D)$ define a linear time-invariant descriptor system. If \emph{(Pa)} holds then $\mathcal{S}_a(x_0)<\infty$ for all $x_0\in\K^n$, $\mathcal{S}_a$ is minimal among all quadratic storage functions and $\mathcal{S}_a(x)=0$ for all $x\in\ker E$. Conversely, if $\mathcal{S}_a(x_0)<\infty$ for all $x_0\in\K^n$ and $\mathcal{S}_a(x_0)=0$ for all $x_0\in\ker E$ then \emph{(Pa)} holds.
\end{proposition}
\begin{proof}
If the system fulfills (Pa) with nonnegative storage function $\mathcal{S}(x)=\tfrac{1}{2}x^HQ^HEx$ for some $Q\in\K^{n\times n}$ then for all $T\geq 0$, $u\in\mathcal{C}^{\infty}([0,T],\K^m)$ and $x_0\in\K^n$ we have that
\[
-\mathcal{S}(x_0) \leq \mathcal{S}(x(T))-\mathcal{S}(x_0)\leq \int_0^T \re y(\tau)^Hu(\tau)d\tau,
\]
which implies 
\[
\mathcal{S}(x_0)\geq\sup_{\substack{u\in \mathcal{C}^{\infty}([0,T],\K^m)\\T\geq 0}}- \int_0^T \re y(\tau)^Hu(\tau)d\tau=\mathcal{S}_a(x_0).
\]
Since $\mathcal{S}$ was arbitrary this implies that $\mathcal{S}_a$ is minimal among all storage functions. Furthermore, by definition $\mathcal{S}_a(x_0)\geq 0$ and $\mathcal{S}(x_0)=0$ for all $x_0\in\ker E$ which implies $\mathcal{S}_a(x_0)=0$.

Conversely, if $\mathcal{S}_a(x_0)<\infty$ holds for all $x_0\in\K^n$, then we show that the function $\mathcal{S}:=\mathcal{S}_a$ satisfies the dissipation inequality \eqref{dissi_ineq}. By Lemma~\ref{lem:Va}, $\mathcal{S}_a(x)=x^HS_ax$ for some $S_a\in\K^{n\times n}$ with $S_a\geq 0$. Furthermore, since $\mathcal{S}_a(x_0)=0$ for all $x_0\in\ker E$ we have $\ker E\subseteq\ker S_a$ and thus $\ran S_a \subseteq \ran E^H$. We can write $\mathcal{S}_a(x)=x^HP_{\ran E^H}S_aP_{\ran E^H}x$. Furthermore, set $Q^H:=
P_{\ran E^H}S_aE^\dagger$ where $E^\dagger$ is the pseudo-inverse of $E$. Then $E^\dagger E=P_{\ran E^H}$ and therefore $Q^HE=P_{\ran E^H}S_a P_{\ran E^H}$. Hence $\mathcal{S}$ is of the particular form required in (Pa). 
Moreover, for all $u\in \mathcal{C}^{\infty}([0,t_2],\K^m)$, $0<t_1\leq t_2$ we have that
\begin{align*}
-\mathcal{S}(x_0)&=\inf_{\substack{u\in \mathcal{C}^{\infty}([0,t_2],\K^m)\\ t_2\geq 0}}\int_0^{t_2} \re y(\tau; x_0)^Hu(\tau)d\tau\\&\leq \int_0^{t_2} \re y(\tau;x_0)^Hu(\tau)d\tau\\&=\int_0^{t_1}\re  y(\tau;x_0)^Hu(\tau)d\tau+\int_0^{t_2-t_1} \re y(\tau;x(t_1))^Hu(\tau+t_1)d\tau,
\end{align*}
where we denote the output of \eqref{DAE} with initial value $x(0)=x_0$ by $y(\cdot;x_0)$ to avoid confusion. Hence, we get
\[
-\mathcal{S}(x_0)-\int_0^{t_2-t_1} \re  y(\tau;x(t_1))^Hu(\tau+t_1)d\tau\leq \int_0^{t_1}\re y(\tau;x_0)^Hu(\tau)d\tau.
\]
Since this holds for all $u(\cdot+t_1)\in \mathcal{C}^{\infty}([0,t_2-t_1],\K^m)$, it holds for the supremum 
\[
\mathcal{S}(x(t_1))-\mathcal{S}(x_0)\leq \int_0^{t_1}\re y(\tau;x_0)^Hu(\tau)d\tau.
\]
Here $u$ is an arbitrary control steering $x_0$ to $x(t_1)$. 
\end{proof}

\section{When does (PR) imply (pH)?}
\label{sec:PR}
In this section, we study whether a pH representation can be obtained from a positive real transfer function or not. This is also of particular interest when one wants to obtain a pH representation from frequency measurements of the transfer function. It was shown in \cite{BGV20} that if the interpolation points are chosen to be spectral zeros then one ends up with a pH realization. However, these spectral zeroes cannot be known in advance if only transfer function measurements are available. One solution proposed in \cite{BGV20} is to construct an intermediate realization from which the spectral zeroes can be computed. The question that arises then is what are the conditions on this intermediate system to end up in pH representation even for an index two pH descriptor system. 

Given a transfer function $\mathcal{T}$ of a descriptor system \eqref{DAE} then $\mathcal{T}$ has a pole of finite order at $\infty$ and if $\nu$ is the index of the pair $(E,A)$ then this order is at most $\nu-1$. Hence, using the Laurent expansion of the entries of $\mathcal{T}$ there exist a sequence of matrices $(M_i)_{i=k}^{-\infty}$ with $M_i\in\dC^{n\times n}$ such that 
\begin{align}
\label{laurent}
\mathcal{T}(s)=\sum_{i=k-1}^{-\infty}M_is^i.
\end{align}
If $\mathcal{T}(s)$ is assumed to be a real rational function, then $M_i\in\dR^{n\times n}$. In the following lemma, we show that for positive real rational functions this representation can be simplified, see also  \cite[Section 5.1]{AndVon73}.
\begin{lemma}
\label{lem:PR_ds}
Let $\Sigma=(E,A,B,C,D)$ be a descriptor system with positive real transfer function $\mathcal{T}$ with Laurent expansion \eqref{laurent} then 
\begin{align}
\label{tf_pr}
\mathcal{T}(s)=M_{1}s+\mathcal{T}_p(s),
\end{align}
holds for some rational function $\mathcal{T}_p(s)$ which fulfills $M_0=\lim\limits_{s\rightarrow\infty}\mathcal{T}_p(s)$. Furthermore, it holds that $M_0+M_0^H\geq 0$ and $M_{1}=M_1^H\geq 0$. Moreover, if the system matrices $E,A,B,C,D$ are real then $M_0$ and $M_1$ real and $\mathcal{T}_p(s)$ is positive real.
\end{lemma}
\begin{proof}
Since $\mathcal{T}$ is a rational function, it has no poles for all $s=i\omega$ with $\lvert\omega\rvert$ sufficiently large. Furthermore, the positive realness and analyticity imply that $\mathcal{T}(i\omega)+\mathcal{T}(i\omega)^H\geq 0$ holds for these values. 

Using the positive realness of $\mathcal{T}$ and \eqref{laurent} for some $k\geq 1$, we conclude for $s=e^{i\varphi}r$ for all $r>0$ and $\varphi>0$ satisfying  $\tfrac{\varphi}{k}\leq \tfrac{\pi}{2}$ that the following holds
\begin{align}
\label{key_formel}
\mathcal{T}(s)+\mathcal{T}(s)^H=M_{k}s^{k}+M_{k}^H\overline{s}^{k}+ \sum_{i=k-1}^{-\infty} M_is^i+M_i^H\overline{s}^i\geq 0.
\end{align}
We consider first the case $k\geq 2$. Considering \eqref{key_formel} for $r>0$ sufficiently large and $\varphi=\tfrac{\pi}{k}$ implies 
\[
M_{k}+M_{k}^H\leq 0.
\]
Furthermore, choosing $\varphi=0$ and $r>0$ sufficiently large in \eqref{key_formel} yields $M_{k}+M_{k}^H\geq 0$. Therefore $M_{k}=-M_{k}^H$ holds. If we consider \eqref{key_formel} with $\varphi=\tfrac{\pi}{2k}$ and $r>0$ sufficiently large leads to $2iM_{k}\leq 0$ and choosing $\varphi=-\tfrac{\pi}{2k}$ yields $-2iM_{k}\leq 0$. Hence we conclude $M_k=0$ and by  repeating this argument (if necessary) we obtain $M_k=\ldots=M_2=0$. 

Therefore \eqref{laurent} holds with $k=1$. Then it remains to prove that $M_1=M_1^H\geq 0$ is satisfied. We decompose $M_1=M_1^++M_1^-$ where $M_1^{\pm}=\frac{1}{2}(M_1\pm M_1^H)$. Then $M_1^+$ is Hermitian and $M_1^-$ is skew-Hermitian and we have $sM_1+(sM_1)^H=s(M_1^++M_1^-)+\overline{s}(M_1^++M_1^-)^H=2(\re s)M_1^++2(\ima s)iM_1^-$. Hence, if we consider $s=i\omega$ and let $\omega\rightarrow\infty$ then this contradicts $\mathcal{T}(s)+\mathcal{T}(s)^H\geq 0$. As a consequence, $M_1^-=0$. Hence $M_1=M_1^H$. If $M_1$ would have a negative eigenvalue with eigenvector $x\in\dC^n$ we obtain a contradiction by considering  $x^H(\mathcal{T}(s)+\mathcal{T}(s)^H)x$ for $s\rightarrow\infty$. This shows that $M_1=M_1^\top\geq 0$. Taking the limit  $\omega\rightarrow\infty$ in the positive realness condition
\[
\mathcal{T}(i\omega)+\mathcal{T}(i\omega)^H\geq 0
\]
we further deduce $M_0+M_0^\top\geq 0$.

If the system matrices are real then clearly $M_0$ and $M_1$ are real. Hence, it remains to conclude that $\mathcal{T}_p(s)$ is positive real. Since $\mathcal{T}(s)$ is real and positive real and $\mathcal{T}_p(i\omega)+\mathcal{T}_p(-i\omega)^\top=\mathcal{T}(i\omega)+\mathcal{T}(-i\omega)^\top\geq 0$ holds for all $\omega$ such that $i\omega$ is not a pole of $\mathcal{T}$, it follows from \cite[Theorem 2.7.2]{AndVon73} that $\mathcal{T}_p(s)$ is positive real. \end{proof}
Note that rational functions $\mathcal{T}$ for which  $\lim_{s\rightarrow\infty}\mathcal{T}(s)$ exists, are called \emph{proper}. Hence we will refer to $\mathcal{T}_p$ in \eqref{tf_pr} as the proper part of a positive real transfer function.

The following example shows that positive realness of arbitrary real rational functions with polynomial growth cannot be concluded from considering the behavior on the imaginary axis alone.
\begin{example}
The function $\mathcal{T}(s)=s^3$ is analytic and hence it has no poles on the imaginary axis. Furthermore, it fulfills  $\mathcal{T}(i\omega)+\mathcal{T}(-i\omega)^\top=0$. However, using the same arguments as in the proof of Lemma~\ref{lem:PR_ds} we find that it is not positive real. Hence Theorem 2.7.2 in \cite{AndVon73} cannot be extended to real rational functions that have a polynomial growth of order larger than two.
\end{example}

Furthermore, from the nonnegativity of the real part of the transfer function on the imaginary axis does not guarantee the positive realness.
\begin{example}
\label{ex:negative_residue}
Consider $\mathcal{T}(s)=-s^{-1}$, which satisfies $\mathcal{T}(s)+\mathcal{T}(s)^H=\tfrac{-2 \re s}{\lvert s\rvert^2}\leq 0$ for all $s\in\dC$ with $\re s\geq 0$ and therefore it is not positive real. Furthermore, a state-space realization is given by $(E,A,B,C,D)=(1,0,B,-B^{-1},0)$ for all scalar $B\neq 0$. Hence the only solution to the KYP inequality
\[
\begin{bmatrix}
-A^HQ-Q^HA&&C^H-Q^HB\\C-B^HQ&&D+D^H
\end{bmatrix}=\begin{bmatrix}
0&-B^{-H}-Q^HB\\-B^{-1}-B^HQ&0
\end{bmatrix}\geq 0
\]
is given by $Q=-\frac{1}{B^2}$. Observe that $\mathcal{T}$ satisfies $\mathcal{T}(i\omega)+\mathcal{T}(i\omega)^H\geq 0$ for all $\omega\neq 0$, but $\mathcal{T}$ has a negative residue it the simple pole at $\omega_0=0$, which is the reason why it is not positive real. Another consequence of this example is the existence of Hermitian solutions to the KYP inequality do not imply that the transfer function of the system is positive real.  
\end{example}

Next, we study the relation of positive realness to the KYP inequality by recalling the following result presented in \cite{FreJ04}.
\begin{proposition}
If $(E,A,B,C,D)$ is a real-valued minimal realization of a positive real transfer function $\mathcal{T}(s)$ with $D+D^\top\geq M_0+M_0^\top$ then \emph{(KYP)} holds.
\end{proposition}

First, observe that one can choose a minimal realization with $D=M_0$. Hence for every positive real transfer function there exists a realization that has a solution to the KYP inequality. However, as we have seen in Section~\ref{sec:Pa_pH} this solution has to fulfill additional requirements if we want to define a pH realization. Furthermore, note that for minimal realizations of positive real transfer functions with index two it was shown in  
\cite[Theorem 4.1]{FreJ04} that (KYP$_{E}$\textbar$\K^n\times\K^m$) is never fulfilled. In particular, we cannot obtain a pH formulation from the solution of this KYP inequality as the following example shows.
\begin{example}
\label{ex:indextwo}
Consider $E=\begin{bmatrix}
1&0\\0&0
\end{bmatrix}$, $A=\begin{bmatrix}
0&-1\\1&0
\end{bmatrix}$, $B=\begin{bmatrix}
0\\1
\end{bmatrix}=C^\top$, $D=0$. Then this is a minimal realization of the positive real transfer function $\mathcal{T}(s)=s$. Here the KYP inequality \eqref{eq:kyp_dae} has the solution $Q=I_2$. However, the inequalities in (KYP$_{E}$\textbar$\K^n\times\K^m$) have no solution.

As a second example consider the descriptor system given by $(E,A,B,C,D)=(0,-1,1,-1,-1)$. Then $\mathcal{T}(s)=0$ which is positive real. However, the KYP inequality \eqref{eq:kyp_dae} has no solution since $D$ is negative. If we would replace $D$ with $M_1=0$ then this KYP inequality becomes solvable.
\end{example} 

As an alternative, we can define a minimal pH realization directly from the transfer function \eqref{tf_pr}. To this end, we consider a minimal realization of the proper part $\mathcal{T}_p(s)$ which is given by 
\[
\mathcal{T}_p(s)=C_p(sE_p-A_p)^{-1}B_p+M_0
\]
where $E_p$ is invertible which follows from the fact that minimal realizations of descriptor systems have no index one blocks in the Kronecker-Weierstra\ss\ form \eqref{eq:wcf}, see e.g.\ \cite[Theorem 6.3]{FreJ04}. 
Then the minimality conditions \eqref{DAE_minimal} and \eqref{DAE_minimal_2} trivially hold which means that $(E_p,A_p,B_p)$ is behaviorally controllable and that $(E_p,A_p,C_p)$ is behaviorally observable. Hence, we know from Proposition~\ref{prop:obs} that there exists invertible $Q_p$ such that $Q_p^HE_p\geq 0$ and 
\[
\begin{bmatrix}
J-R && G-P\\(G-P)^H&&D
\end{bmatrix}:=\begin{bmatrix}
A_pQ_p^{-1}&&B_p\\C_pQ_p^{-1}&&D_p
\end{bmatrix}.
\]
Furthermore, a minimal pH realization of $sM_{1}$ is given by 
\begin{align}
\label{sys: inf}
\begin{split}
E_{\infty}&=\begin{bmatrix}M_{1}&0\\0&0\end{bmatrix},~~ A_{\infty}=\begin{bmatrix}0&-I_m\\I_m&0\end{bmatrix}\in\dR^{2m\times 2m},~~ C_{\infty}=[0,I_m]=B_{\infty}^\top\in\dR^{m\times 2m},\\~~ D_{\infty}&=0\in\dR^{m\times m}.
\end{split}
\end{align}
Indeed 
\[
C_{\infty}(sE_{\infty}-A_{\infty})^{-1}B_{\infty}=C_{\infty}\begin{bmatrix}0&&I_m\\ -I_m&&sM_{1}\end{bmatrix}B_{\infty}=sM_{1}
\]
and the minimality conditions \eqref{DAE_minimal} and \eqref{DAE_minimal_2} can be verified easily. Furthermore, the system (\ref{sys: inf}) is pH with $Q_{\infty}=I_{2m}$.

In the following lemma, we show that we can combine the two pH systems from the proper and the non-proper part to obtain a minimal pH system realization of a positive real transfer function. 
\begin{lemma}
\label{lem:sum_minimal}
Let $\Sigma_i=(E_i,A_i,B_i,C_i,D_i)$, $i=1,2$, be descriptor systems which fulfill \emph{(pH)} for some $Q=Q_1$ and $Q=Q_2$ respectively.  
Then the system $\Sigma_+$ given by
\begin{align*}
E_+&:=\begin{bmatrix}
E_1&0\\0&E_2
\end{bmatrix},\quad  A_+:=\begin{bmatrix}
A_1&0\\0&A_2
\end{bmatrix},\quad B_+:=\begin{bmatrix}
B_1\\B_2
\end{bmatrix},\quad C_+:=\begin{bmatrix}
C_1&C_2
\end{bmatrix},\\ D_+&:=D_1+D_2
\end{align*}
has the transfer function $\mathcal{T}_+=\mathcal{T}_1+\mathcal{T}_2$ and fulfills \emph{(pH)} with $Q=\di(Q_1,Q_2)$. 
\end{lemma}
\begin{proof}
If systems $\Sigma_i$, $i=1,2$, fulfill (pH) then there exists  $J_i,R_i,Q_i\in\K^{n_i\times n_i}$, $G_i,P_i\in\K^{n_i\times m_i}$, and $S_i,N_i\in\K^{m_i\times m_i}$ such that 
\begin{align*}
\begin{bmatrix}
A_i&B_i\\C_i&D_i
\end{bmatrix}&=\begin{bmatrix}(J_i-R_i)Q_i&G_i-P_i\\(G_i+P_i)^HQ_i&S_i+N_i\end{bmatrix},\quad Q_i^HE_i=E_i^HQ_i\geq 0,\\
	\Gamma_i &:= \begin{bmatrix}
		J_i & G_i \\
		-G_i^H& N_i\end{bmatrix}
		= - \Gamma_i^H,\quad 
 W_i := \begin{bmatrix}
	R_i & P_i\\
	P_i^H & S_i
\end{bmatrix} =W_i^H \geq 0.
\end{align*}
By setting $J_+ = \di(J_1,J_2)$, $R_+ =\di(R_1,R_2)$, $Q_+=\di(Q_1,Q_2)$, $G_+=\di(G_1,G_2)$, $P_+=\di(P_1,P_2)$, $S_+=\di(S_1,S_2)$ and $N_+=\di(N_1,N_2)$ it holds that
\begin{align*}
\begin{bmatrix}
A_+&B_+\\C_+&D_+
\end{bmatrix}&=\begin{bmatrix}(J_+-R_+)Q_+&G_+-P_+\\(G_++P_+)^HQ_+&S_++N_+\end{bmatrix},\quad Q_+^HE_+=E_+^HQ_+\geq 0,\\
	\Gamma_+ &:= \begin{bmatrix}
		J_+ & G_+ \\
		-G_+^H& N_+\end{bmatrix}
		= - \Gamma_+^H,\quad 
 W_+ := \begin{bmatrix}
	R_+ & P_+\\
	P_+^H & S_+
\end{bmatrix} =W_+^H \geq 0.
\end{align*}
Hence $\Sigma_+$ fulfills (pH).
\end{proof}

In summary this means that for behaviorally controllable and observable descriptor systems with positive real transfer functions we can obtain a pH realization via
\begin{align*}
E_{pH}&:=\begin{bmatrix}
E_p&0&0\\0&M_{1}&0\\0&0&0
\end{bmatrix},& A_{pH}&:=\begin{bmatrix}
A_p&0&0\\0&0&-I_m\\0&I_m&0
\end{bmatrix}, & B_{pH}&:=\begin{bmatrix}
B_p\\0\\I_m\end{bmatrix},\\ C_{pH}&:=\begin{bmatrix}
C_p&0&I_m\end{bmatrix},& D_{pH}&:=M_0, \quad
Q:=\begin{bmatrix}
Q_p&0\\0&I_{2m}
\end{bmatrix}.
\end{align*}
This is then already a pH descriptor system in the so called staircase form which was studied recently in \cite{BeaGM19}.

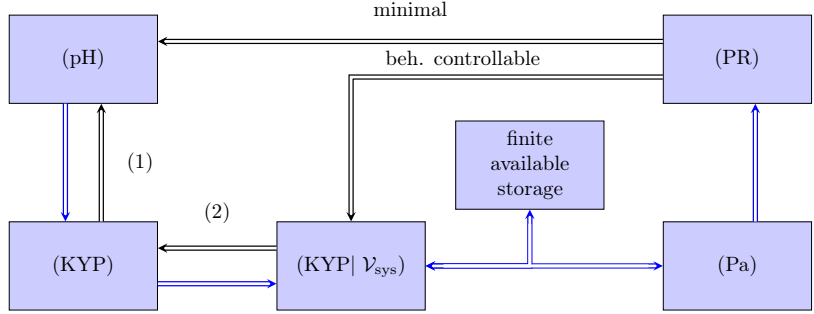
\begin{figure}
    \centering
    \scalebox{0.78}{
\begin{tikzpicture}
        \node[block] (a) {(pH)};
        \node[block, below =2cm of a]   (b){(KYP)};
        \node[block, right =2cm of b]   (c){(KYP$\mid\Vs$)};
        \node[block, right =4cm of c]   (d){(Pa)};
        \node[block, above =2cm of d]   (e){(PR)};
        \node[block, above right = 0.2cm and 0.5cm of c ]   (f){finite\\available\\storage};
        
        \draw[->,semithick,double,double equal sign distance,>=stealth, color=blue] ([xshift=-2ex]a.south) -- ([xshift=-2ex]b.north);
        \draw[->,semithick,double,double equal sign distance,>=stealth] ([xshift=2ex]b.north) -- ([xshift=2ex]a.south) node[midway,right = 2 ex]{(1)};
        \draw[->,semithick,double,double equal sign distance,>=stealth, color=blue] ([yshift=-2ex]b.east) -- ([yshift=-2ex]c.west);
        \draw[->,semithick,double,double equal sign distance,>=stealth] ([yshift=2ex]c.west) -- ([yshift=2ex]b.east) node[midway,above = 2 ex]{(2)};
        \draw[<->,semithick,double,double equal sign distance,>=stealth, color=blue] (c.east) -- (d.west)node[midway,below = 2 ex]{};
        \draw[<->,semithick,double,double equal sign distance,>=stealth,color=blue] (c.east) -| (f.south)node[midway,above right = 2 ex and 2ex]{};
         \draw[-,semithick, white,line width=1.4pt, shorten >= 7pt] ([xshift=2ex]c.east) -- (d.west);
        \draw[-,semithick, white,line width=1.4pt, shorten >= 7pt] ([xshift=2ex]c.east) -| (f.south);
        \draw[->,semithick,double,double equal sign distance,>=stealth, color=blue] ([xshift=2ex]d.north) -- ([xshift=2ex]e.south) ;
        \draw[->,semithick,double,double equal sign distance,>=stealth] ([yshift=2ex]e.west) -- ([yshift=2ex]a.east)node[midway,above = 2 ex]{minimal};
        \draw[->,semithick,double,double equal sign distance,>=stealth, color=black] ([yshift=-2ex]e.west) -| (c.north)node[midway,above right= 0.5ex and 3ex]{beh. controllable~~};
        \end{tikzpicture}
        }
       \caption{Overview of the relationship between (pH), (KYP), (Pa) and (PR) for descriptor systems with $(E,A)$ regular. The implications with additional assumption are colored black and the one without in blue. 
       \\
       (1) $\ker Q \subseteq \ker C\cap\ker A$\\
       (2) behaviorally observable and has index at most one}
    \label{fig:overview2}
\end{figure}
The connections of (pH), (KYP), (Pa) and (PR)  for descriptor systems is summarized in Figure~\ref{fig:overview2}.

\section{Conclusion}
In this paper, we studied conditions for the equivalence between passive, positive real and port-Hamiltonian descriptor systems, as well as their relation to the solvability of generalized KYP inequalities. The conditions on the equivalence already available in the literature were either validated or relaxed and counterexamples were also presented in the cases where the equivalence does not hold. In addition, we considered special cases: index one descriptor systems and KYP inequalities on a subspace which are shown to be equivalent to the finiteness of the available storage. Finally, we focused on conditions to obtain a port-Hamiltonian system from either a passive system, a positive real transfer function or the solution of the KYP inequality. 

As future work, the analysis conducted in this note can be extended in several directions. Namely, one can study more general passivity properties such as \emph{cyclo passivity} \cite{SchJ21} which allows for possibly negative storage functions. In addition, we only focused on continuous time systems. A similar study could be conducted for discrete time systems. This could also help in finding a definition analogous to (pH) for this class of systems. Finally, pH systems are dissipative in the sense of \cite{Wil72} with respect to a specific supply rate $w(x,u)=\re u^H(Cx+Du)$. The analysis conducted in this note could also be extended to systems which are dissipative with respect to other quadratic supply rates such as the \emph{scattering supply rate} which is given by $w(x,u)=\|u\|^2-\|Cx+Du\|^2$ and closely related to the transfer functions being bounded real.

\subsection*{Declaration of competing interest}

The authors declare that they have no known competing financial interests or personal
relationships that could have appeared to influence the work reported in this paper.

\bibliography{biblio}
\bibliographystyle{plain}

\begin{appendix}
\section{Brief recap on the definition of system space}
\label{sec:systemspace}

The goal of this section is to recall properties of the system space $\Vs$ and to show that the two definitions of the system space given in the references \cite{ReiV19} and \cite{ReiRV15} are equivalent.

The system space $\Vs$ was characterized in \cite[Lemma 3.7]{ReiV19} using initial values of state and input functions $x$ and  $u$ which solve the equation \eqref{DAE} and whose derivatives $x^{(k)}$ and $u^{(k)}$ of arbitrary order $k\in\mathbb{N}$ exist. These solutions will be called  \emph{smooth} in this section. 
\begin{lemma}
\label{lem:smooth}
Let $[E, A,B]$ be given. Then the following holds:
\begin{itemize}
\item[(i)] For all $(x_0,
u_0) \in \Vs$, there exists some smooth solution
$(x, u)$ of $[E, A,B]$ with $x(0) =x_0$ and $u(0)=u_0$. In particular, $x_0 \in\mathcal{V}_{\rm diff}$.
\item[(ii)] For all $x_0\in \mathcal{V}_{\rm diff}$ there exists some smooth solution $(x, u)$
of $[E,A,B]$ with $Ex(0)=Ex_0$. In particular, there exists some $(x_{01}
u_{01}) \in\Vs$ with $Ex_0=Ex_{01}$.
\end{itemize}
\end{lemma}

As a consequence of this lemma we obtain the following inclusion
\begin{align}
\label{eq:vsys_smooth}
\Vs\subseteq \Vs^\infty:=\{(x(0),u(0))~:~ \text{$x,u$ are smooth solutions of \eqref{DAE}}\}.
\end{align}

In the following, we will show that $\Vs=\Vs^\infty$ holds.

We consider for invertible $S,T\in\K^{n\times n}$ and $F\in\K^{m\times n}$ the \emph{feedback equivalent} system
\begin{align}
\label{eq:fe_system}    
(E_F,A_F,B_F):=(SET,S(A+BF)T,SB).
\end{align}
It is easy to see that the solution $(x,u)$ of a system \eqref{DAE} and the solution $(x_F,u_F)$ of its feedback equivalent system \eqref{eq:fe_system} are connected via the following invertible linear transformation
\[
\begin{bmatrix}
x(t)\\ u(t)
\end{bmatrix}=\begin{bmatrix}
    T&0\\FT&I_m
    \end{bmatrix}\begin{bmatrix}
x_F(t)\\ u_F(t)
\end{bmatrix}.
\]

Hence for a comparison of the system spaces, we can study some feedback equivalent system.

If $(E,A)$ is regular then it was shown e.g.\ in \cite[Proposition 2.12]{IlcR17}, see also \cite[Theorem 3.2]{ByeGM97}, that there exists $F\in\K^{m\times n}$ and invertible $S,T\in\K^{n\times n}$ such that $(E_F,A_F,B_F)$ given by \eqref{eq:fe_system} fulfills 
\begin{align}
\label{eq:fe_form}    
(E_F,A_F,B_F)=\left(\begin{bmatrix}
I_{n_1}&0&0\\0&0&E_{23}\\ 0&0&E_{33}
\end{bmatrix},\begin{bmatrix}
A_1&0&0\\0&I_{n_2}&0\\ 0&0&I_{n_3}
\end{bmatrix},\quad\begin{bmatrix}
B_1\\ B_2\\0
\end{bmatrix}\right)
\end{align}
where $E_{33}$ is nilpotent. The matrices \eqref{eq:fe_form} lead to the differential algebraic equation 
\begin{align}
\label{fef_dae}
\tfrac{d}{dt}\begin{bmatrix}
I_{n_1}&0&0\\0&0&E_{23}\\ 0&0&E_{33}
\end{bmatrix}\begin{pmatrix}x_1\\x_2\\x_3\end{pmatrix}=\begin{bmatrix}
A_1&0&0\\0&I_{n_2}&0\\ 0&0&I_{n_3}
\end{bmatrix}\begin{pmatrix}x_1\\x_2\\x_3\end{pmatrix}+\begin{bmatrix}
B_1\\ B_2\\0
\end{bmatrix}u.
\end{align}
Since $E_{33}$ is nilpotent, every solution of \eqref{fef_dae} fulfills $x_3=0$. Therefore, 
\[
x_2(t)=-B_2u(t),\quad x_1(t)=e^{A_1t}x_{1}^0+\int_0^te^{A_1(t-2)}B_1u(s)ds
\]
for some initial value  $x_{1}^0\in\K^{n_1}$. 
This implies that there exists a  smooth solution of \eqref{fef_dae} if and only if there exist $x_{1}^0\in\K^{n_1}$ and $u^0\in\K^{m}$ such that the following holds 
\[
\begin{pmatrix}
x(0)\\u(0)
\end{pmatrix}=\begin{pmatrix}
x_1^0\\B_2u^0\\0\\u^0
\end{pmatrix}.
\]
Hence for the system \eqref{eq:fe_form} we obtain
\begin{align}
\label{vsyssmooth_formula}
\Vs^\infty=\ran\begin{bmatrix}
I_{n_1}&0\\0&B_2\\0&0\\0&I_m
\end{bmatrix}.
\end{align}

We continue now with showing that  $\Vs=\Vs^\infty$ holds for the system \eqref{eq:fe_form}. Assume that $\Vs$ is a proper subspace of $\Vs^\infty$. Then there exist $(x_0,u_0)\in\Vs^\infty$ which is orthogonal to $\Vs$. By definition, there exists a smooth solution $(x,u)$ satisfying $(x(0),u(0))=(x_0,u_0)$.

The definition of $\Vs$ implies that $(x,u)\in L^2_{loc}(\dR,\Vs)$ and therefore $(x_0,u_0)^\top(x(t),u(t))=0$ holds for almost all $t\in\dR$. Since the considered function is continuous, it is identically zero. On the other hand, we have $0=(x_0,u_0)^\top(x(0),u(0))=\|(x_0,u_0)\|^2>0$ which is the desired contradiction. From this we conclude that the following holds
\[
\Vs=\Vs^\infty.
\]

A slightly different notion of system space was considered in \cite{ReiRV15} where instead of absolute continuity of $Ex$ with $\tfrac{d}{dt}Ex\in L^2_{loc}(\dR,\K^n)$ it was assumed that $x$ is absolutely continuous with $\tfrac{d}{dt}x\in L^2_{loc}(\dR,\K^n)$. However, it was observed in \cite[Proposition 3.3]{ReiRV15} that their system space coincides with the space on the right hand side in \eqref{vsyssmooth_formula} and hence with $\Vs^\infty$. This allows us to conclude that the $\Vs$ considered in \cite{ReiV19} coincides with the system space introduced in \cite{ReiRV15}.

A geometric characterization of $\Vs$ in terms of \emph{Wong sequences} was obtained in \cite{ReiRV15}. There it was shown that the following sequence terminates after finitely many steps
\begin{align}
\label{wong}
\mathcal{V}_{k+1}:=[A,B]^{-1}([E,0]\mathcal{V}_k),\, k\geq 1,\quad \mathcal{V}_0:=\K^{n}\times \K^m,
\end{align}
where the pre-image of $[A,B]$ is used, and that the resulting subspace coincides with $\Vs$.

\end{appendix}

\end{document}